\documentclass[reqno]{amsart}%
\usepackage{amssymb}
\usepackage{amsfonts}
\usepackage{amsmath}
\usepackage{cancel}
\usepackage[numbers]{natbib}
\usepackage{xcolor}
\usepackage{verbatim}
\usepackage{mathtools}
\usepackage{graphicx}%
\usepackage[colorlinks=true, pdfstartview=FitV, linkcolor=blue, citecolor=blue, urlcolor=blue]{hyperref}
\usepackage{tikz}
\usepackage{float}
\usepackage{enumitem}

\allowdisplaybreaks[4]

\newtheorem{theorem}{Theorem}[section]
\theoremstyle{plain}

\newtheorem{conjecture}[theorem]{Conjecture}

\newtheorem{definition}[theorem]{Definition}
\newtheorem{example}[theorem]{Example}

\newtheorem{lemma}[theorem]{Lemma}

\newtheorem{proposition}[theorem]{Proposition}
\newtheorem{remark}[theorem]{Remark}

\numberwithin{equation}{section}


\begin{document}
\title[]{A Positive Mass Type Theorem for a Singular Toroidal Static Model}

\author{Zhixin Wang}
\address{Department of Mathematics, Shanghai Jiao Tong University, Shanghai, 201100}
\email{jhin@sjtu.edu.cn}

\begin{abstract}
We study positive mass and related Brown--York type inequalities for three-dimensional manifolds modeled on a static space with flat toroidal slices. Using inverse mean curvature flow, we first derive an inequality relating the asymptotic geometry at infinity to the size of an interior singularity. This inequality admits a natural interpretation as a positive mass type theorem. We then combine this global inequality with a Shi--Tam type construction to obtain a Brown--York type inequality for boundaries isometric to flat tori. In contrast to the Schwarzschild setting, our argument does not extend directly to general boundary surfaces, and we identify several obstructions to such a generalization. Finally, we discuss examples illustrating the relationship between negative mass, the strength of interior singularities, and the topology of the interior.

\end{abstract}
\maketitle

\tableofcontents

\section{Introduction}
The notion of mass in Riemannian geometry originates from the ADM mass of asymptotically flat manifolds, introduced by Arnowitt, Deser, and Misner in the study of isolated gravitating systems \cite{ADM1962}. A fundamental result in this direction is the positive mass theorem (PMT) of Schoen--Yau \cite{SchoenYau1979,SchoenYau1981} and Witten \cite{Witten1981}, which asserts, under suitable assumptions, that an asymptotically flat manifold with nonnegative scalar curvature has nonnegative ADM mass, with equality precisely for Euclidean space. Recently, the non-spin case was extended to dimensions up to \(19\) by Bi--Hao--He--Shi--Zhu \cite{BiHaoHeShiZhu2026}, and to arbitrary dimensions by Brendle--Wang \cite{BrendleWang2026}.
 The theory was later extended to the asymptotically hyperbolic setting. In particular, notions of mass and corresponding positive mass theorems were established by Wang \cite{Wang2001AHMass} and by Chruściel--Herzlich \cite{ChruscielHerzlich2003}, with the hyperbolic space playing the role of the zero-mass rigidity model. Both the Euclidean space and hyperbolic space are static spaces, which is defined as follows.

\begin{definition}[Static Manifold]
Let $(M, g)$ be an $n$-dimensional Riemannian manifold. It is said to be \emph{static} if it admits a non-trivial smooth function $V \in C^\infty(M)$ (referred to as a \emph{static potential}) satisfying:
\begin{equation}\label{static-eqn}
    \Delta V g - \nabla^2 V + V \operatorname{Ric} = 0,
\end{equation}
where $\nabla^2 V$, $\Delta V$, and $\operatorname{Ric}$ denote the Hessian, the Laplacian, and the Ricci curvature tensor of $(M, g)$, respectively. $(M,g,V)$ is also referred to static triple.
\end{definition}

More generally, static manifolds arise naturally as extremal geometries in mass minimization problems.  One might naturally hope to extend positive mass type results to manifolds with toroidal slices. Two important static model spaces arise: the hyperbolic manifold of constant sectional curvature \(-1\) with toroidal slices, and the Horowitz--Myers soliton. These two models exhibit different mass phenomena. Under suitable topological assumptions, positive mass theorems are available for ALH manifolds with toroidal ends; see, for instance, \cite{ChruscielGallowayNguyenPaetz2018}. On the other hand, unlike the spherical case, absolute positivity cannot be expected in general. The Horowitz--Myers soliton is a complete static ALH manifold with flat toroidal conformal infinity and strictly negative mass \cite{HorowitzMyers1998}. This led Horowitz and Myers to conjecture that, after fixing the asymptotic geometry, their metric should instead provide the sharp lower bound for the mass.

Recently, Brendle and Hung proved the corresponding sharp lower bound in dimensions \(3\leq n\leq 7\) and subsequently established the rigidity of the equality case \cite{BrendleHungHM,BrendleHungRigidity}, with the Horowitz--Myers soliton as the ground state. The static uniqueness of the Horowitz--Myers soliton was studied earlier by Galloway, Surya, and Woolgar \cite{GallowaySuryaWoolgar2003}, and more recently further investigated by the author \cite{Wang2025Neumann}. Thus, in the toroidal setting, both positive and negative mass phenomena naturally arise.

In this paper, we consider the model
\begin{equation}\label{model}
   (M=\mathbb{R}_+\times T^{n-1}, g=dr^2+r^{4/n}h)
\end{equation}
where $h$ is a flat metric on $T^{n-1}$. Under the change of variable 

\begin{equation}\label{r-x-variable}
    r=\frac{n-2}{n}x^{n/(n-2)}
\end{equation}
 the metric becomes
\begin{equation}\label{model2}
    g=x^{4/(n-2)}(dx^2+ (\frac{n-2}{n})^{4/n}h)
\end{equation}
This model can be regarded as a generalized Kottler space \cite{ChruscielSimon2001}, and is static with potential function $r^{(2-n)/n}$. Based on this model, we introduce the following asymptotic condition.
\begin{definition}
We say $(M,g)$ satisfies \textbf{Condition C} provided that there exists open sets $U_1\subset U_2$ so that $U_1$ is diffeomorphic to $(0,\epsilon)\times T^{n-1}$ and $M\setminus U_2$ is diffeomorphic to $(1/\epsilon,\infty)\times T^{n-1}$. Using $\big(r,\vec{\theta}=(\theta_1,\cdots,\theta_{n-1})\big)$ on these two ends, we assume the metric satisfies

\begin{equation}\label{g-asym0}
    g=\begin{cases}
        &dr^2+r^{4/n}(h_1+E_1) \qquad r\in (0,\epsilon)\\
        &dr^2+r^{4/n}(h_2+E_2)\qquad r\in (\frac{1}{\epsilon},\infty)
    \end{cases}
\end{equation}
where $h_1$ and $h_2$ are two flat metric on $T^{n-1}$ and $r=\frac{n-2}{n}x^{n/(n-2)}$ as (\ref{model2}) and the error tems $E_1,E_2$ satisfy
\begin{equation}
    \begin{aligned}\label{g-asym}
        |E_1|+r|\partial_\alpha E_1|+r^2|\partial_\alpha\partial_\beta E_1|=O(r)\\
    |E_2|+r|\partial_\alpha E_2|+r^2|\partial_\alpha\partial_\beta E_2|=o(1)
    \end{aligned}
\end{equation}
where $\alpha,\beta\in\{r,\theta_1,\cdots,\theta_{n-1}\}$.
Throughout this paper, $S(r)$ or $S_r$ will denote the level sets of $r$.
\end{definition}
Based on this model, in this paper we can prove the following theorem:

\begin{theorem}\label{thm-main1}
    Let $(M^3,g)$ be a smooth three dimensional and satisfies \textbf{Condition C}. Suppose $R\geq 0$ and $M$ is irreducible (every embedded $S^2\subset M$ bounds a $3$-ball), then
    \begin{equation}\label{thm-main1-ineq}
        |h_1|\geq |h_2|
    \end{equation}
where $|h|=\operatorname{Area}(T^2,h)$. Furthermore equality holds iff $h_1$ is isometric to $h_2$ and $g$ takes form (\ref{model}).
\end{theorem}
The inequality \eqref{thm-main1-ineq} can be viewed as a positive mass type theorem from two perspectives: its relation to the limiting Hawking mass and to a relative static flux motivated by the Hamiltonian mass functional \cite{ChruscielHerzlich2003}. We discuss these interpretations in Section~\ref{why-mass}. Moreover, our theorem is restricted to three-dimensional manifolds because the proof relies on the monotonicity formula for inverse mean curvature flow, which is specific to dimension three.

Static metrics are also naturally associated with critical points of weighted total mean curvature. To see this, let \((M,\bar g,V)\) be a static manifold and let \(\Omega\subset M\) be a compact domain with boundary \(\Sigma\). Let $g$ be a metric on $\Omega$ so that $g|_\Sigma=\bar{g}|_\Sigma$ and consider the functional

\begin{equation}
    \mathcal H_V(g)
    :=
    \int_{\Sigma} V H_g\,dA,
\end{equation}
The static equation (\ref{static-eqn}), together with the variational formula for scalar curvature and mean curvature, shows that \(\bar g\) is a critical point of \(\mathcal H_V\) among scalar-curvature-preserving deformations inducing the same metric on \(\Sigma\). Actually, $\bar{g}$ is thought to maximize this functional in a lot of cases.

Remarkably, such a local boundary inequality is closely related to global positive mass inequalities and, in several important settings, can be derived from an appropriate global positive mass or Penrose-type theorem.

The prototype is the theorem of Shi and Tam \cite{Shi2002PositiveMT}, who established the positivity of the Brown--York mass for surfaces admitting convex isometric embeddings into Euclidean space by applying the Riemannian positive mass theorem. Replacing the Euclidean reference by a nonflat static background, Lu and Miao \cite{Lu2017MinimalHA} obtained a weighted Brown--York type inequality for surfaces isometrically embedded into the spatial Schwarzschild manifold, using the Riemannian Penrose inequality as the corresponding global input. More recently, Alaee, Hung, and Khuri \cite{AlaeeHungKhuri2025} established an analogous positivity theorem for the static Brown--York mass of tori modeled on the zero-mass Kottler manifold. Thus, deriving Brown--York type inequalities from suitable global positive mass theorems has become a standard and effective strategy in the study of quasi-local mass.

In \cite{Lu2017MinimalHA}, the theorem was stated only for $n\leq 7$,and this restriction has now been removed by the recent work of Bi and Zhu \cite{BiZhu2026}, who proved the Riemannian Penrose inequality in arbitrary dimensions.

Based on this type of argument, we can prove a Brown-York type inequality from \textbf{Theorem \ref{thm-main1}}.
\begin{theorem}
\label{thm-brownyork}
Let $(\Omega,g)$ be an irreducible Riemannian $3$-manifold with boundary
\(        \partial\Omega=S\cup\Sigma,
   \)
where $g$ is smooth away from $S$. Assume that $R\geq 0$, and
\begin{enumerate}
\item $\Sigma$ is diffeomorphic to $T^2$, has $H>0$, and $(\Sigma,g|_{\Sigma})$ admits an isometric embedding into $(\bar M,\bar g)$ as a coordinate torus $\{r_0\}\times T^2$ for some $r_0>0$, where $(\bar M,\bar g)$ is of the form \eqref{model} for some flat metric $\bar h$ on $T^2$;

\item $S$ is diffeomorphic to $T^2$, and there exists an open neighborhood $U\subset\Omega$ of $S$, diffeomorphic to $(0,\epsilon)\times T^2$, such that, near $\{0\}\times T^2$, the metric $g$ satisfies the asymptotic expansion given by the first line of \eqref{g-asym0} and \eqref{g-asym} for some flat metric $h$ on $T^2$.
\end{enumerate}
Then
\begin{equation}\label{thm-ineq-brownyork-1}
    \int_{\Sigma} H\,dA_g
    \leq
    \frac{4}{3}A((\Sigma,g|_{\Sigma}))^{1/4}|h|^{3/4}.
\end{equation}

\end{theorem}

It is worth mentioning that, unlike the Brown-York type inequalities in \cite{Shi2002PositiveMT}\cite{Lu2017MinimalHA}, where the hypersurface \(\Sigma\) may be chosen from a broad class of convex surfaces, our result is restricted to the case where \((\Sigma,g|_{\Sigma})\) is a flat torus. The proof there may be divided into three main steps. Interestingly, none of these steps seems to extend directly to general surfaces in our setting. We discuss these obstructions and provide explicit examples in Section~\ref{subsection-Failure-in-General}.

Another interesting feature of the model~\eqref{model} is its connection with negative mass. As will be explained later, the model may be viewed as carrying negative mass at infinity. In the asymptotically flat setting, Robbins~\cite{Robbins2010} and Bray--Jauregui~\cite{BrayJauregui2013} studied negative ADM mass through the notion of zero area singularities and introduced a corresponding notion of singular mass. The idea is that \textbf{negative mass may lead to the development of interior
singularities, and their strength is controlled by
the extent to which the geometric inequality fails}. Interestingly, their framework can also be adapted to the singular behavior arising in our model~\eqref{model}, and Theorem~\ref{thm-main1} admits a natural interpretation within this setting. 

This phenomenon admits an extension to the asymptotically locally hyperbolic setting with toroidal conformal infinity. Since this is not the main focus of the present paper, we state only a simple version, although several further generalizations are possible.

\begin{theorem}\label{thm-ALH-ZAS}
Let $(M^3=\mathbb{R}_+\times T^2,g)$ be a Riemannian manifold satisfying the following assumptions:
\begin{enumerate}
\item There exists a neighborhood $U$ of $\{0\}\times T^2$ such that
\begin{equation}\label{thm-ALH-g-asmp}
g|_U
=
\frac{dr^2}{r^2-\frac{2m_1}{r}}
+r^2h_1,
\end{equation}
where $m_1<0$ and $h_1$ is a flat metric on $T^2$.

    \item As $r\to\infty$, the metric $g$ is asymptotically locally hyperbolic with conformal infinity $(T^2,h_2)$, where $h_2$ is a flat metric on $T^2$. Let $\mu$ denote the corresponding mass aspect function.
\end{enumerate}
If $\mu\leq 0$ and $R\geq -6$, then
\begin{equation}
    m_{\mathrm{reg}}(\{0\}\times T^2)
    \leq
    \frac{1}{4(4\pi)^{3/2}}
    \sup_{T^2}\mu\,|h_2|^{3/2}.
\end{equation}

\end{theorem}

The notions of asymptotically locally hyperbolic manifolds, mass aspect function, and regular mass $m_{\mathrm{reg}}$ will be recalled in later sections. We emphasize that the assumption $\mu\leq0$ is essential for the argument. As mentioned earlier, inverse mean curvature flow
does not in general recover the total mass at infinity
in the ALH setting, since the limiting Hawking mass
may not agree with the asymptotic mass; see \cite{Neves2010IMCF}. When the mass aspect function is nonpositive, however, the inverse mean curvature flow still yields useful control of the mass at infinity, as shown by Lee and Neves \cite{Lee2013ThePI}. This allows the preceding argument to extend to the toroidal ALH setting under the sign assumption on $\mu$.

Moreover, since the level sets in our model are tori, additional phenomena involving topology and negative mass arise, just as the Horowitz-Myers soliton. These issues will be discussed in Section~\ref{mass-negative}.

\noindent\textbf{Acknowledgements}

The author is partially supported by China Postdoctoral Science Fundation 2025T180838.

The paper is organized as follows. In Section~2, we collect the necessary preliminaries and explain why our main Theorem~\ref{thm-main1} can be regarded as a positve mass type theorem. In Section~3, we use inverse mean curvature flow to prove Theorem~\ref{thm-main1}. In Section~4, we combine this inequality with a Shi--Tam type argument to obtain a Brown--York type inequality when \(\Sigma\) is a flat torus. We also explain why this argument does not extend directly to general \(\Sigma\). Finally, in Section~5, we give further interpretations and applications of our results, including their relation to relative positive mass, negative mass phenomena, and the Horowitz--Myers soliton. We also include a short proof for Theorem~\ref{thm-ALH-ZAS}.

\section{Preliminaries}

\subsection{Geometric Setup and Basic Computations}
\begin{proposition}\label{prop-curvature-computation}
    Suppose $g$ satisfies \textbf{Condition C}, let $X_i=\frac{1}{r^{2/n}}\partial_i$. As $r\rightarrow \infty$,
    \begin{equation}
    \begin{aligned}\label{sectional-curvature}
       R(\partial_r,X_i,X_i,\partial_r) &= \frac{2(n-2)}{n^2 r^2} +o(\frac{1}{r^2})\\[1ex]
     R(X_i,X_j,X_j,X_i)&= - \frac{4}{n^2 r^2} +o(\frac{1}{r^2})\quad i\neq j\\[1ex]
        \operatorname{Ric}_g &= \frac{2(n-1)(n-2)}{n^2 r^2} dr^2 - \frac{2(n-2)}{n^2} r^{\frac{4-2n}{n}} h+o(\frac{1}{r^2})
    \end{aligned}
\end{equation}
In particular, there's no $o(\frac{1}{r^2})$ term if the metric is exactly of the model (\ref{model}).
\end{proposition}

\begin{proof}

We use the conformal model so that $g=x^{4/(n-2)}(dx^2+ (\frac{n-2}{n})^{4/n}h+E_2)$ where $E$ is error term. Recall the formula of sectional curvature under conformal change $g=u^{4/(n-2)}\bar{g}$
\begin{equation}
\begin{aligned}
R(X, Y, Y, X) &= u^{\frac{4}{n-2}} \Bigg\{ \bar{R}(X, Y, Y, X) \\
&\quad - \frac{2}{(n-2)u} \Big[ |X|^2_{\bar{g}} \bar{\nabla}^2 u(Y, Y) + |Y|^2_{\bar{g}} \bar{\nabla}^2 u(X, X) - 2 \bar{g}(X, Y) \bar{\nabla}^2 u(X, Y) \Big] \\
&\quad + \frac{2n}{(n-2)^2 u^2} \Big[ |X|^2_{\bar{g}} Y(u)^2 + |Y|^2_{\bar{g}} X(u)^2 - 2 \bar{g}(X, Y) X(u) Y(u) \Big] \\
&\quad - \frac{4}{(n-2)^2 u^2} |\bar{\nabla} u|^2_{\bar{g}} \Big[ |X|^2_{\bar{g}} |Y|^2_{\bar{g}} - \bar{g}(X, Y)^2 \Big] \Bigg\}
\end{aligned}
\end{equation}
Consider the end at infinity. Let $r=\frac{n-2}{n}x^{n/(n-2)}$, then
\begin{equation}
    r\partial_r
    =
    \frac{n-2}{n}x\partial_x,
    \qquad
    r^2\partial_r^2
    =
    \frac{(n-2)^2}{n^2}x^2\partial_x^2
    -
    \frac{2(n-2)}{n^2}x\partial_x.
\end{equation}
and the second line of (\ref{g-asym}) becomes
\begin{equation}
      |E_2|+x|\partial_\alpha E_2|+x^2|\partial_\alpha\partial_\beta E_2|=o(1)
\end{equation}
with $\alpha,\beta\in \{x,\theta_1,\cdots, \theta_{n-1}\}$. By the linearization for curvature operator (\cite{Topping2006LecturesOT} for example), we have
\begin{equation}
\bar{R}(X,Y,Z,W)=o(\frac{1}{x^2})
\end{equation}
for all unit $X,Y$. Besides

\begin{equation}
    \begin{aligned}
    \frac{1}{x}\bar{\nabla}^2x(\partial_{\alpha},\partial_{\beta})&=-\frac{1}{x}\bar{\Gamma}_{\alpha\beta}^{\gamma}\partial_\gamma x=o(\frac{1}{x^2})\\
    |\bar{\nabla}u|&=1+o(1)
    \end{aligned}
\end{equation}
Therefore
\begin{equation}
    \begin{aligned}
R(\partial_x, \partial_{i},\partial_i,\partial_x) &= \frac{2}{(n-2)}(\frac{n-2}{n})^{4/n}x^{\frac{8-2n}{n-2}}+o(x^{\frac{8-2n}{n-2}})\\
R(\partial_{i}, \partial_{j},\partial_j,\partial_i) &=-\frac{4}{(n-2)^2}(\frac{n-2}{n})^{8/n}x^{\frac{8-2n}{n-2}}+o(x^{\frac{8-2n}{n-2}})
\end{aligned}
\end{equation}
And the proposition follows.
\end{proof}

\begin{proposition}
    The metric $(M,g)$ in (\ref{model}) satisfies the static equation with potential function $\frac{1}{x}$ or $(\frac{1}{r})^{(n-2)/n}$
\end{proposition}
\begin{proof}
Set $V=r^{-(n-2)/n}$,and direct computation gives
\begin{equation}
\nabla^2V
=
\frac{2(n-1)(n-2)}{n^2}
r^{(2-3n)/n}\,dr^2
-
\frac{2(n-2)}{n^2}
r^{(6-3n)/n}\,h.
\end{equation}

\end{proof}

\subsection{Mass in Asymptotically Flat and ALH Manifolds}\leavevmode\newline

We briefly recall the definitions of mass for asymptotically flat and
asymptotically locally hyperbolic manifolds.
\begin{definition}
An $n$-dimensional Riemannian manifold $(M^n,g)$ is called
\emph{asymptotically flat} if, outside a compact set, there exists a
coordinate system $(x^1,\ldots,x^n)$ such that
\[
    g_{ij}=\delta_{ij}+O_2(|x|^{-\tau}),
    \qquad
    \tau>\frac{n-2}{2},
\]
together with the usual integrability assumption on the scalar
curvature. Its ADM mass is defined by
\begin{equation}\label{eq:ADM-mass}
    m_{\mathrm{ADM}}(g)
    =
    \frac{1}{2(n-1)\omega_{n-1}}
    \lim_{R\to\infty}
    \int_{S_R}
    \left(
        \partial_jg_{ij}-\partial_i g_{jj}
    \right)\nu^i\,dA_0,
\end{equation}
where $\omega_{n-1}=|\mathbb S^{n-1}|$ and $\nu$ is the Euclidean
outward unit normal. In particular, when $n=3$, the constant in front
is $1/(16\pi)$. 
\end{definition}
We refer to \cite{Bartnik1986} for the basic theory of
the ADM mass. The fundamental example for ADM mass is the Schwartzchild metric.
\begin{example}
For $m>0$,
    \begin{equation}\label{def-Schwartzchild}
        M^{n}_m=
\mathbb{R}^{n}\setminus
\left\{
|x|<
\left(\frac{m}{2}\right)^{\frac{1}{n-2}}
\right\},\quad 
\bar{g}_m=\left(1+\frac{m}{2}|x|^{2-n}
\right)^{\frac{4}{n-2}}g_E\\
    \end{equation}
    $(M,\bar{g})$ is a static with potential function $V=(1-\dfrac{m}{2}|x|^{2-n})(1+\dfrac{m}{2}|x|^{2-n})^{-1}$, and the parameter $m$ is the ADM mass of the Schwartzchild metric. And the boundary
       $\partial M_m^n=\{|x|=
\left(\frac{m}{2}\right)^{\frac{1}{n-2}}\big\}$
    is a minimal surface.
\end{example}
We next recall the asymptotically hyperbolic case. 
\begin{definition}\label{def-ALH}
    Let
$(N^{n-1},h)$ be either the round sphere or a flat torus, and set
\[
    k=
    \begin{cases}
        1, & (N,h)=(\mathbb S^{n-1},g_{\mathbb S^{n-1}}),\\
        0, & (N,h)=(T^{n-1},h_{\mathrm{flat}}).
    \end{cases}
\]
We say that an end a Riemannian manifold $(M,g)$ is asymptotically locally hyperbolic (ALH), with conformal
infinity $(N,h)$, if near infinity there is a defining function
$x\to0$ such that
\begin{equation}\label{eq:AH-expansion}
    g
    =
    \frac{1}{x^2}
    \left[
        dx^2+
        \left(1-\frac{kx^2}{4}\right)^2h
        +\frac{1}{n}\kappa x^n
        +O(x^{n+1})
    \right],
\end{equation}
where $\kappa$ is a symmetric $2$-tensor on $N$, and $x$ is called boundary defining function. The mass aspect
function is
\[
    \mu=\operatorname{tr}_h\kappa,
\]
 the Wang--Chru\'sciel--Herzlich mass, denoted by $m_{ALH}$, is
\begin{equation}\label{eq:AH-mass}
    m_{ALH}(g)
    =
    \frac{1}{2(n-1)\omega_{n-1}}
    \int_N \operatorname{tr}_h\kappa\,d\mu_h.
\end{equation}
\end{definition}

Thus, the above definition applies to both the spherical and toroidal
cases \cite{ChruscielHerzlich2003}\cite{Wang2001AHMass}. Below we provide with two examples.
\begin{example}
Let $(T^{n-1},h)$ be a flat torus. The $n$-dimensional toroidal
Kottler metric (see, for example, \cite{Birmingham1999,GallowaySuryaWoolgar2003}) is
\begin{equation}\label{def-Kottler}
    \bar{g}_m
    =
    \frac{dr^2}{f_m(r)}+r^2h,
    \qquad
    f_m(r)
    =
    r^2-\frac{2m}{r^{n-2}}.
\end{equation}
It has constant scalar curvature $-n(n-1)$. Apply the transformation $r=x^{-1}\left(1+\frac{m}{2}x^n\right)^{\frac{2}{n}}$, one direct sees  
\begin{equation}
    \begin{aligned}
        \kappa=2mh; &\quad \mu=2m(n-1)\\
        m_{ALH}(\bar{g}_m)&=\frac{m}{\omega_{n-1}}|h|
    \end{aligned}
\end{equation}

\end{example}
Another fundamental example is the
Horowitz--Myers soliton \cite{HorowitzMyers1998,GallowayWoolgar2015}.

\begin{example}\label{example-HM}
\begin{equation}\label{g-HM-asym}
        M=B^2\times T^{n-2},\quad
    g_{\mathrm{HM}}=
    \frac{1}{x^2}
    \left[
        dx^2+
        \left(1+\frac{x^n}{4}\right)^{4/n}
        \left(
            \left(
            \frac{1-x^n/4}{1+x^n/4}
            \right)^2d\xi^2
            +h_F
        \right)
    \right]
\end{equation}
where $\xi$ has period $4\pi/n$ so that the metric extends smoothly
across $r=1$, and $h_F$ is a flat metric on $T^{n-2}$. Its static potential is $V$ satisfying $x=4^{1/n}\big(V^{n/2}-\sqrt{V^n-1}\big)^{2/n}$ and
$R_{g_{\mathrm{HM}}}=-n(n-1)$.

Thus
\begin{equation}\label{g-HM-asym-2}
    \kappa=-(n-1)\,d\xi^2+h_F,
    \qquad
    \operatorname{tr}_{h}\kappa=-1,
    \qquad
    h=d\xi^2+h_F.
\end{equation}
Therefore
\begin{equation}
    m_{\mathrm{AH}}(g_{\mathrm{HM}})
    =
    -\frac{1}{2(n-1)\omega_{n-1}}|h|<0.
\end{equation}
\end{example}

\subsection{Mass Interpretations of the Main Inequality}\label{why-mass}\leavevmode\newline

We next explain why \textbf{Theorem \ref{thm-main1}} can naturally be interpreted as a mass related theorem. Consider a 3-dimensional manifold $(M^3,g)$ satisfying \textbf{Condition C}. For a closed hypersurface $\Sigma$, its Hawking mass is defined as

\begin{equation}\label{eq:hawking-mass}
    m_H(\Sigma)
    :=
    \sqrt{\frac{|\Sigma|}{16\pi}}
    \left(
        \frac{\chi(\Sigma)}{2}
        -
        \frac{1}{16\pi}
        \int_\Sigma H^2\,dA
    \right),
\end{equation}

Consider the $S_r = \{r\} \times T^2$ for $(M,g)$. As $r\to\infty$, we have
\begin{equation}\label{eq:slice-computation}
    H_{S_r}=\frac{4}{3r}+o(r^{-1}),\quad
    d\mu_{S_r}=r^{4/3}\bigl(1+o(1)\bigr)\,d\mu_{h_2},\quad
    |S_r|=r^{4/3}\bigl(|h_2|+o(1)\bigr).
\end{equation}
Since $\chi(S_r)=\chi(T^2)=0$, the topology-normalized Hawking
mass satisfies
\begin{equation}
\begin{aligned}
    m_H(S_r)
    &=
    -\sqrt{\frac{|S_r|}{16\pi}}
    \frac{1}{16\pi}
    \int_{S_r}H_{S_r}^2\,d\mu_{S_r}\\
    &=
    -\sqrt{\frac{r^{4/3}(|h_2|+o(1))}{16\pi}}
    \frac{1}{16\pi}
    \left(
        \frac{16}{9}r^{-2/3}|h_2|
        +o(r^{-2/3})
    \right)\\
    &=
    -\frac{|h_2|^{3/2}}{36\pi^{3/2}}+o(1).
\end{aligned}
\end{equation}
Therefore,
\begin{equation}\label{2-lim-Hawking}
    \lim_{r\to\infty}m_H(S_r)
    =
    -\frac{|h_2|^{3/2}}{36\pi^{3/2}}.
\end{equation}

The limit of the Hawking mass is closely related to the mass at infinity. In the asymptotically flat setting, the Hawking mass of large coordinate spheres converges to the ADM mass \cite{Huisken2001TheIM}. An analogous relation holds for canonical coordinate slices in the asymptotically hyperbolic setting, with the appropriate normalization, although convergence to the total mass may fail along inverse mean curvature flow \cite{Neves2010IMCF}. These relations motivate the following:
\begin{definition}\label{def-mass-F}
    Let $(M,g)$ satisfies \textbf{Condition C}, then its mass at infinity is defined as
    \begin{equation}
        m_F(g)= -\frac{|h_2|^{3/2}}{36\pi^{3/2}}.
    \end{equation}
\end{definition}

There is a second viewpoint based on the Hamiltonian mass flux associated
with a static reference metric, as developed in
\cite{ChruscielHerzlich2003}; see also \cite{HuangJang2022}.
In the standard asymptotic settings, this flux gives the familiar
ADM and Wang--Chru'sciel--Herzlich mass invariants.

Let $(M,\bar g,V)$ be a static triple and let $g$ be another metric.
Set $e=g-\bar g$. Following
\cite{ChruscielHerzlich2003,HuangJang2022}, the relative mass is defined by

\begin{equation}\label{eq:relative-mass}
m_{\bar g}(g,V)
=
\lim_{r\to\infty}
\int_{\Sigma_r}
\left[
V\left(
(\operatorname{div}_{\bar g}e)(\nu)
-\nu(\operatorname{tr}_{\bar g}e)
\right)
+
(\operatorname{tr}_{\bar g}e)\nu(V)
-
e(\nabla^{\bar g}V,\nu)
\right]dA_{\bar g}.
\end{equation}
In the standard setting, the Hamiltonian flux defines a relative mass under suitable asymptotic decay assumptions. In our setting, however, the two model metrics differ already at leading order, so we do not claim that the following quantity is the standard relative mass. Rather, we use the same flux expression as a natural static comparison between nearby model metrics.
We first take
\begin{equation}
    \bar g=dr^2+r^{4/3}h_1,
    \qquad
    V=r^{-1/3},
    \qquad
    g=dr^2+r^{4/3}h_2 .
\end{equation}
for some flat metric $h_1,h_2$ on $T^2$. Writing $e=r^{4/3}(h_2-h_1)$, we have
\begin{equation}
\begin{aligned}
    \operatorname{tr}_{\bar g}e
    &=\operatorname{tr}_{h_1}(h_2-h_1),\\
    (\operatorname{div}_{\bar g}e)(\partial_r)
    &=-\frac{2}{3r}\operatorname{tr}_{h_1}(h_2-h_1),\\
    \partial_rV&=-\frac{V}{3r},
    \qquad
    e(\nabla^{\bar g}V,\partial_r)=0
\end{aligned}
\end{equation}
Therefore,
\begin{equation}\label{eq:model-relative-mass}
    m_{\bar g}(g,V)
    =
    -\int_{T^2}
    \operatorname{tr}_{h_1}(h_2-h_1)\,d\mu_{h_1}.
\end{equation}

In particular, if $h_2=c\,h_1$, then
\begin{equation}
    m_{\bar g}(g,V)
    =
    2\bigl(|h_1|-|h_2|\bigr).
\end{equation}

For general $h_2$, let $h_t$, $t\in[0,1]$, be a smooth path of flat
metrics joining $h_1$ to $h_2$. Applying the same computation to
\begin{equation}
    \bar g_t=dr^2+r^{4/3}h_t,
    \qquad
    e_t=r^{4/3}\dot h_t,
\end{equation}
gives
\begin{equation}
    m_{\bar g_t}(e_t,V)
    =
    -\int_{T^2}
    \operatorname{tr}_{h_t}\dot h_t\,d\mu_{h_t}
    =
    -2\frac{d}{dt}|h_t|.
\end{equation}
This quantity measures the infinitesimal mass type deficit in the direction
$e_t$ relative to the background metric $\bar g_t$. We therefore define
\begin{equation}
\begin{aligned}
    m_{\mathrm{rel}}(h_1,h_2)
    &:=
    \frac12\int_0^1
    m_{\bar g_t}(e_t,V)\,dt\\
    &=
    |h_1|-|h_2|.
\end{aligned}
\end{equation}
In particular, the resulting quantity depends only on the endpoints
\(h_1\) and \(h_2\), and not on the choice of the path \(h_t\).
Hence the \textbf{Theorem \ref{thm-main1}} may be interpreted as the
nonnegativity of the relative mass-type quantity.

\section{Proof of Theorem~\ref{thm-main1}}
In this section we use Inverse Mean Curvature Flow (IMCF) to prove Theorem~\ref{thm-main1}. The classical IMCF is a parabolic flow $F$ that starts with a closed strictly mean convex hypersurface $\Sigma_0$ and evolves the hypersurface by the inverse of its mean curvature:

\begin{equation}\label{def-IMCF}
    \frac{\partial}{\partial_t}F=\frac{1}{H}\nu
\end{equation}
where $\nu$ is the outer normal vector. Alternatively, define arriving function $u$ by $u^{-1}(t)=\Sigma_t$, then (\ref{def-IMCF}) is equivalent to a degenerated elliptic equation of $u$
\begin{equation}\label{def-IMCF-u}
    \operatorname{div}\bigg(\frac{\nabla u}{|\nabla u|}\bigg)=|\nabla u|
\end{equation}

Furthermore, a smooth function $u$ is called sub solution of IMCF provided taht 

\begin{equation}\label{def-sub-IMCF-u}
    \operatorname{div}\bigg(\frac{\nabla u}{|\nabla u|}\bigg)\geq |\nabla u|
\end{equation}
and its called super solution if we have $\leq$ instead. If we assume $|\nabla u|>0$, then a sub-solution gives a flow that moves faster than the classical IMCF while super-solution moves slower.

The existence of classical IMCF has been intensively studied. For example, starting from a star-shaped, strictly mean convex hypersurface in $\mathbb{R^n}$, the classical IMCF exists for all time and remains strictly mean convex \cite{Gerhardt1990FlowON,Urbas1990OnTE}. Similar results are also obtained for anti-de-Sitter Schwartzchild space \cite{Lu2016InverseCF} and general warped prodct space satisfying certain restriction \cite{Zhou2016InverseMC}. However, in many natural settings---such as when the exterior region is not diffeomorphic to $\Sigma_0 \times [0, \infty)$---the smooth flow develops singularities and terminates in finite time. This makes it necessary to consider a weak version of the flow. 

Heuristically, the weak formulation introduced in \cite{Huisken2001TheIM} avoids singularity formation through a jumping mechanism. In this process, there are countably many times $t$ when the hypersurface $\Sigma_t$ immediately ``jumps'' to its outermost area-minimizing envelope.

Analytically, the weak formulation is established by finding suitable variational principles for the flow. The initial value problem of the weak IMCF is defined as follows:

\begin{definition}[\cite{Huisken2001TheIM}, p. 365]\label{def-weakIMCF}
Let $M$ be a connected, complete, non-compact manifold, and let $E_0 \subset M$ be a bounded smooth domain. A locally Lipschitz function $u$ on $M$ is called a \textbf{weak solution (subsolution, super solution respectively)} of the IMCF with initial condition $E_{T_0}$ if:
\begin{enumerate}
    \item $E_{T_0} = \{u < T_0\}$,
    \item For any compact set $K \Subset M \setminus E_{T_0}$ and any locally Lipschitz function $v$ ($v\leq u$, $v\geq u$ respectively) such that $\{u \neq v\} \subset K$, we have
    \begin{equation}
        \int_K \big(|\nabla u| + u|\nabla u|\big) \leqslant \int_K \big(|\nabla v| + v|\nabla u|\big).
    \end{equation}
\end{enumerate}

Besides, we call open set $E$ to be a minimizing hull if $E$ minimizes area on the outside (in $\Omega$), that is

\begin{equation}
    |\partial^*E\cap K|\leq |\partial^*F\cap K|
\end{equation}
for any $F$ containing $E$ such that $F\setminus E\Subset \Omega$, F, and any compact set $K$ containing $F \setminus E$. We call it strictly minimizing hull if equality implies $F\cap \Omega=E\cap \Omega \quad a.e.$
\end{definition}

The weak solution $u$ is called \textbf{proper} if the sublevel set $\{ s\leq u \leq t\}$ is bounded for all $s<t$. Besides, throughout this section, let $E_t\coloneqq \{u<t\}$, and let $E_t^+=\{u\leq t\}$. 

For our model metric (\ref{model}), if we starts from a level set of $r$ we get the solution
\begin{equation}
u=\frac{4}{3}\log r+T
\end{equation}
where $T$ is some fixed constant. If $(M,g)$ satisfies \textbf{Condition C}, given any $\epsilon$ and we define

\begin{equation}\label{sub-super-solution}
    u_{\pm,\epsilon}=(\frac{4}{3}\pm\epsilon)\log r
\end{equation}
 there exists $r_0$ so that $u_{-,\epsilon}$ is a sub-solution and $u_{+,\epsilon}$ is a super-solution for $r\geq r_0$. And the flow they corresponds to is $(S(\exp((\frac{4}{3}\pm \epsilon)^{-1}(t)))_{0<t<\infty}$. The existence of this subsolution goes to infinity and prevents IMCF from jumping to infinity instantly, and thus guarantees the existence of IMCF from a general mean-convex initial surface.

\begin{theorem}\label{IMCF-existence} Let $M$ be a complete, connected Riemannian $3$-manifold without boundary. Suppose there exists a proper subsolution. Then for any nonempty, connected, precompact, smooth open set minimizing hull $E_0$ in $M$, there exists a proper, locally Lipschitz solution $u$ of \textbf{Definition \ref{def-weakIMCF}} with initial condition $E_0$, which is unique on $M \setminus E_0$. Furthermore, we have the following

    \begin{enumerate}
\item For all $t \geqslant 0, \Sigma_t:=\partial E_t$ is a $C^{1, \alpha}$ hypersurface. The weak mean curvature of $\Sigma_t$ (denoted by $H$ ) exists for all $t$ and is equal to $|\nabla u|$ almost everywhere for almost every $t$, and $H>0$ almost everywhere for almost every $t$.

\item The area of $\Sigma_t$ grows exponentially: $\left|\Sigma_t\right|=e^t\left|\Sigma_0\right|$.

\item For all $t>0, E_t$ is connected and $M \backslash E_t$ has no compact connected component.

\item Let $R$ be the scalar curvature of $M$. Then

\begin{equation}\label{IMCF-monotonicity}
\begin{aligned}
\int_{\Sigma_{t_2}} H^2 \leqslant \int_{\Sigma_{t_1}} H^2- & \int_{t_1}^{t_2} \int_{\Sigma_s}\left[2 \frac{\left|\nabla_{\Sigma_s} H\right|^2}{H^2}+|\mathring{II}|^2\right] d s \\
& +\int_{t_1}^{t_2}\left[4 \pi \chi\left(\Sigma_s\right)-\int_{\Sigma_s} R-\frac{1}{2} \int_{\Sigma_s} H^2\right] d s, \quad \forall 0 \leqslant t_1<t_2
\end{aligned}
\end{equation}
where $\mathring{II}$ is the traceless second fundamental form.
\item $u$ satisfies the gradient estimate
\begin{equation}\label{IMCF-gradient-estimate}
|\nabla u(p)| \le \sup_{\partial E_{0} \cap B_s(p)} H_+ + \frac{C(n)}{s} , \quad \text{a.e. } x \in M \setminus E_0,
\end{equation}
\textit{for each $0 < s \le \sigma(p)$, where $\sigma(p) > 0$ is defined in \textbf{Definition \ref{sigma-radius}}.}

\item $E_t^+$ is strictly minimizing hull for each $t\geq 0$, $\{u<t\}$ is minimizing hull for each $t>0$. Besides, let $\Sigma_t^+=\partial E_t^+$, then
\begin{equation}
    \Sigma_t\rightarrow \Sigma_{t_0} \text{ as }t\nearrow t_0,\quad \Sigma_t\rightarrow \Sigma_{t_0}^+ \text{ as }t\searrow t_0
\end{equation}
\end{enumerate}

\end{theorem}

The theorem also applies in our singular setting, since singularity is enclosed within the initial hypersurface.

\begin{definition}\label{sigma-radius}
 For any $p \in M$, define $\sigma(p) \in (0, \infty]$ to be the supremum of radii $r$ such that $B_r(p) \subset\subset M$,
\begin{equation}
\text{Rc} \ge -\frac{1}{100nr^2} \qquad \text{in } B_r(p),\label{sigma-radius-ric}
\end{equation}
and there exists a $C^2$ function $f$ on $B_r(p)$ such that
\begin{equation}\label{sigma-radius-d}
f(p) = 0, \qquad f \ge d_p^2 \qquad \text{in } B_r(p),
\end{equation}
yet
\begin{equation}\label{sigma-radius-hess}
|\nabla f| \le 3d_p \quad \text{and} \quad \nabla^2 f \le 5g \qquad \text{on } B_r(p),
\end{equation}
where $d_p$ is the distance to $p$.
\end{definition}
\begin{remark}
    In the original definition in \cite{Huisken2001TheIM}, it's required $\nabla^2 f \le 3g$ instead of $\nabla^2 f \le 5g$ in (\ref{sigma-radius-hess}). But if we replace $3$ by $5$, the same argument still work.
\end{remark}

We next derive a lower bound for $\sigma(p)$ that will be used in the gradient estimate \eqref{IMCF-gradient-estimate}.

\begin{lemma}\label{lemma-sigma-estimate}
For $p=(r_0,\vec{\theta}_0)$ with $r_0$ sufficiently large,
\begin{equation}\label{estiamte-sigma}
\sigma(p)\geq \frac{1}{1000}r_0.
\end{equation}
\end{lemma}
\begin{proof}

Fix a $p=(r_0,\theta_0)$ and $r_0$ large and let $A\left(r_1,r_2 \right)=\left\{r_1<r<r_2\right\}$ be the annular region. Since $\lim_{r\rightarrow \infty}|\nabla r|=1$, for another point $q=(r,\vec{\theta})\in B(p,\frac{1}{1000}r_0)$, we have 

\begin{equation}
    |r(q)-r(p)|<\frac{1001}{1000}d(p,q)\leq \frac{2}{1000}r_0
\end{equation}
Therefore by (\ref{sectional-curvature})
\begin{equation}
 \quad|Sec | \leq \frac{1}{r^2} \leq(\frac{1006}{1000})^2\frac{1}{r_0^2}
\end{equation}
which verifies (\ref{sigma-radius-ric}). Consider $d_p^2$, and by Hesian commarison, away from the cut-locus

\begin{equation}
\begin{aligned}
\nabla^2 d_p & \leq \frac{1006}{1000 r_0} \operatorname{coth}\left(\frac{1006}{1000r_0} d_p\right) \cdot g \\
&  \leq \frac{1}{d_p}\left(1+\frac{1006}{1000r_0} d_p\right)g \leq \frac{501}{500d_p}g
\end{aligned}
\end{equation}
We use $t\coth(t)\leq 1+t$ in the second line. As a result 
\begin{equation}
\nabla^2 d_p^2=2 d_p D^2 d_p+2 d(d_p) \otimes d(d_p)<5g
\end{equation}
Thus $d_p^2$ satisfies the estimates in
\eqref{sigma-radius-hess} away from the cut locus. By the smooth
approximation technique of Greene--Wu
\cite{Greene1979CinftyA}, one can smooth $d_p^2$ from above to obtain
a smooth function $f$ on $B(p,r_0/1000)$ and (\ref{sigma-radius-d})(\ref{sigma-radius-hess}) is preserved.

\end{proof}

The following comparison property follows from
\cite[Theorem~2.2]{Huisken2001TheIM}.

\begin{lemma}\label{IMCF-comparison}

Let $(M,g)$ satisfies \textbf{Condition C}, $u_{\pm,\epsilon}$ are weak super(sub) solutions. Then $u_{\pm,\epsilon}$ satisfies the comparison principle in the following sense: Let $u$ is a weak solution of IMCF with initial condition $E_0$., 

\begin{enumerate}
    \item if $\{p\in M:u(p)\leq t\}\subset \{p\in M:u_{-,\epsilon}(p)\leq t'\}$, then $\{p\in M:u(p)\leq t+\tau\}\subset \{p\in M:u_{-,\epsilon}(p)\leq t'+\tau\}$ for any $\tau\geq 0$
    \item  if $\{p\in M:u_{+,\epsilon}(p)\leq t'\}\subset\{p\in M:u(p)\leq t\}$, then $\{p\in M:u_{+,\epsilon}(p)\leq t'+\tau\}\subset \{p\in M:u(p)\leq t+\tau\}$ for any $\tau\geq 0$
\end{enumerate}
\end{lemma}

\begin{lemma}(Geroch Monotonicity) Given a $C^1$ surface $\Sigma$, define the Hawking mass to be
    \begin{equation}\label{def-Hawking}
        m_H(\Sigma)=-\sqrt{\frac{|\Sigma|}{(16\pi)^3}}\int_{\Sigma}H^2dA
    \end{equation}
    Let $u$ be a weak IMCF from \textbf{Theorem \ref{IMCF-existence}}, and $\Sigma_t=\partial\{u<t\}$, then
    \begin{equation}\label{Geroch-monotonicity}
\begin{aligned}
m_H(\Sigma_{t_2})-m_H(\Sigma_{t_1})\geq  & \frac{1}{(16\pi)^{3/2}}\bigg(\int_{t_1}^{t_2}|\Sigma_{s}|^{1/2} \int_{\Sigma_s}\left[2 \frac{\left|\nabla_{\Sigma_s} H\right|^2}{H^2}+|\mathring{II}|^2dA\right] d s \\
& +\int_{t_1}^{t_2}|\Sigma_s|^{1/2}\left[-4 \pi \chi\left(\Sigma_s\right)+\int_{\Sigma_s} RdA\right] d s\bigg), \quad \forall 0 \leqslant t_1<t_2
\end{aligned}
\end{equation}
\end{lemma}
\begin{proof}
    The inequality follows directly from (2)(4) of \textbf{Theorem \ref{IMCF-existence}}.
\end{proof}

\begin{proposition}\label{global-existence}
Under the assumptions of Theorem 1.3, there exists a function
$u$ on $M$ such that for every $T\in\mathbb{R}$, $u-T$ is a weak solution of IMCF
with initial condition $\{u<T\}$, and 
\begin{equation}\label{global-existence-ends-estimate}
    u\rightarrow -\infty \text{ as }r\rightarrow 0;\quad u\rightarrow \infty \text{ as }r\rightarrow \infty
\end{equation}
 Moreover, for almost every $t$,
\begin{equation}
    \chi(\Sigma_t)\leq0.
\end{equation}
Consequently, $m(\Sigma_t)$ is monotonically non-decreasing and
\begin{equation}\label{estimate-Hawking}
    m(\Sigma_t)
    \geq
    -\frac{1}{36\pi^{3/2}}|h_1|^{3/2}
    \qquad
    \forall\,t.
\end{equation}
\end{proposition}

The proof is divided into three parts. First, we show that the slices
$S(r)=\{r=\mathrm{const}\}$ are outer minimizing for sufficiently small
$r$, which allows us to construct subsolutions and supersolutions going to infinity. Next, we apply Theorem~\ref{IMCF-existence} to obtain weak IMCFs starting
from $S(\epsilon_0/i)$ to obtain
a sequence $u_i$. The subsolutions and supersolutions provide uniform
local bounds for $u_i$, which allow us to pass to a global weak limit.
Finally, we prove \eqref{estimate-Hawking} using Geroch monotonicity.

\begin{proof}
\noindent\textit{(i)}
We can choose $\epsilon_0>0$ sufficiently small so that
$\{S(r)\}_{0<r<\epsilon_0}$ is a strictly mean-convex foliation of
$\{0<r<\epsilon_0\}$. Define
\begin{equation}
    i(\epsilon)
    :=
    \inf
    \bigg\{
        |\Sigma|:
        \Sigma \text{ separates the two ends of }M
        \text{ and }
        \Sigma\subset\{r\geq\epsilon\}
    \bigg\}
    >0.
\end{equation}
The positivity follows from the geometry of $(M,g)$ at infinity.
By standard geometric measure theory, the infimum is achieved by a
minimizing surface, denoted by $N(\epsilon)$, which is minimal away
from its contact set with $S(\epsilon)$.

Choose $\epsilon_1<\epsilon_0$ such that $|S(\epsilon_1)|<i(\epsilon_0)$.
Then, for every $\epsilon<\epsilon_1$,
\begin{equation}
    i(\epsilon)
    \leq |S(\epsilon)|
    <i(\epsilon_0),
\end{equation}
and hence $N(\epsilon)$ cannot be entirely contained in
$\{r\geq\epsilon_0\}$.

We next show that $N(\epsilon)$ is contained in
$\{\epsilon\leq r<\epsilon_0\}$ for all sufficiently small $\epsilon$.
Otherwise, $N(\epsilon)$ intersects $S(\epsilon_0)$ at some point $p$.
We may choose $\rho>0$, independent of $\epsilon$, such that
\begin{equation}
    B(p,\rho)\cap S(\epsilon)=\emptyset
\end{equation}
for every $\epsilon<\epsilon_1$, and such that the monotonicity formula
for minimal surfaces applies in $B(p,\rho)$. Since $N(\epsilon)$ is
minimal in this ball, there exists a constant
$C=C(M,g,\epsilon_0,\rho)>0$ such that
\begin{equation}
    |N(\epsilon)|
    \geq
    |N(\epsilon)\cap B(p,\rho)|
    \geq
    C.
\end{equation}

Choose $\epsilon_2<\epsilon_1$ sufficiently small so that $
    |S(\epsilon_2)|<C$. Then, for every $\epsilon<\epsilon_2$,
   $ N(\epsilon)
    \subset
    \{\epsilon\leq r<\epsilon_0\}$, since otherwise $C>|S(\epsilon)|\geq |N(\epsilon)|\geq C$, contradiction.

Let $p\in N(\epsilon)$ be a point where $r$ attains its maximum. If
$r(p)>\epsilon$, then $N(\epsilon)$ is tangent to $S(r(p))$ at $p$ and
lies on its inner side. Therefore
\begin{equation}
    H\bigl(N(\epsilon)\bigr)
    \geq
    H\bigl(S(r(p))\bigr)
    >0,
\end{equation}
This contradicts the fact that
$N(\epsilon)$ is minimal near $p$. Therefore $r(p)=\epsilon$, and hence
$N(\epsilon)=S(\epsilon)$.

\noindent\textit{ii)} By \eqref{g-asym},
\begin{equation}\label{IMCF-initial-H}
H(S(r))=\frac{4}{3r}+O(1).
\end{equation}
Therefore, there exists a constant $C>0$ such that
\begin{equation}
t-T=v_{\pm}(r)=\frac{4}{3}\log (r\pm Cr^2)
\end{equation}
give a supersolution and a subsolution, respectively, for sufficiently small $r$. Indeed, the speed of the flow given by $v_+$ is
\begin{equation}
\frac{\partial r}{\partial t}
=\frac{3}{4}\frac{r+Cr^2}{1+2Cr}
< \frac{1}{H}
=\frac{3}{4}\frac{r}{1+O(r)}.
\end{equation}
Thus $S(v_+^{-1}(t))$ moves more slowly than IMCF and hence $v_+$ is a supersolution. The argument for $v_-$ is similar.

Next, let $\delta_i=\frac{\epsilon_2}{i}$, and let $\bar{u}$ be the weak IMCF with initial surface $S(\epsilon_2)$. Using $v_-$ as comparison for large $r$, we have 
\begin{equation}\label{bar-u-infinity}
    \bar{u}\rightarrow \infty \text{ as }r\rightarrow \infty
\end{equation}
Define
\begin{equation}
u_{\pm,i}(p)=
\begin{cases}
v_{\pm}(r(p))-v_{\pm}(\delta_i), & \delta_i\leq r(p)\leq \epsilon_2,\\
\bar{u}(p)+v_{\pm}(\epsilon_2)-v_{\pm}(\delta_i), & r(p)\geq \epsilon_2.
\end{cases}
\end{equation}

Let $u_i$ be the weak IMCF with initial surface $S(\delta_i)$. Since $u_{-,i}$ is a weak subsolution in $\{\delta_i\leq r\leq \epsilon_2\}$ and both flows start from $S(\delta_i)$, the comparison principle gives
\begin{equation}\label{3-glueing-comparison}
    u_i\geq u_{-,i} \qquad \text{for } \delta_i\leq r\leq \epsilon_2.
\end{equation}
In the region outside $S(\epsilon_2)$, we apply the comparison principle again to $u_i$ and $u_{-,i}$, using \eqref{3-glueing-comparison} as the initial ordering on $S(\epsilon_2)$. This yields $u_i\geq u_{-,i}$ outside $S(\epsilon_2)$. The same argument applies to $u_{+,i}$, and thus
\begin{equation}
    u_{+,i}\geq u_i\geq u_{-,i}
    \qquad\text{on }\{r\geq\delta_i\}.
\end{equation}

Furthermore, for $p\in{\delta_i<r<\epsilon_2}$,
\begin{equation}\label{gloabl-solution-middle-1}
\begin{aligned}
u_{+,i}(p)-u_{-,i}(p)
&=
\frac{4}{3}\log\left(\frac{1+Cr}{1-Cr}\right)
-
\frac{4}{3}\log\left(\frac{1+C\delta_i}{1-C\delta_i}\right)\
&\leq
\frac{4}{3}\log\left(\frac{1+C\epsilon_2}{1-C\epsilon_2}\right).
\end{aligned}
\end{equation}
Since outside $S(\epsilon_2)$, both $u_{\pm}$ are continued with the same function $\bar{u}$, the same estimate continues to hold
\begin{equation}
u_{+,i}(p)-u_{-,i}(p)\leq C(\epsilon_2)
\end{equation}
for all $i$ and all $p$ where they are defined. Moreover,
\begin{equation}\label{gloabl-solution-middle-2}
|u_{+,i}(p)-u_{+,i}(q)|
=
\begin{cases}
|v_+(r(p))-v_+(r(q))|, & \delta_i<r(p),r(q)<\epsilon_2,\\
|v_+(\epsilon_2)-v_+(r(p))+\bar{u}(q)|, & r(p)<\epsilon_2,\quad r(q)>\epsilon_2,\\
|\bar{u}(p)-\bar{u}(q)|, & r(p),r(q)>\epsilon_2.
\end{cases}
\end{equation}
Using \eqref{gloabl-solution-middle-1} and \eqref{gloabl-solution-middle-2}, we obtain
\begin{equation}\label{gloabl-solution-middle-3}
\begin{aligned}
|u_i(p)-u_i(q)|
&\leq |u_{+,i}(p)-u_{-,i}(q)|
+|u_{+,i}(q)-u_{-,i}(p)|\\
&\leq 2C(\epsilon_2)+2|u_{+,i}(p)-u_{+,i}(q)|\\
&\leq 2C(\epsilon_2)+C(p,q).
\end{aligned}
\end{equation}

By \eqref{IMCF-gradient-estimate}, for every compact set
$K\Subset M$ there exists $s>0$ such that, for all sufficiently large
$i$, $B_s(p)\cap S(\delta_i)=\emptyset$
for every $p\in K$. Hence
\begin{equation}
|\nabla u_i|\leq C(K)
\end{equation}
uniformly on $K$. Together with \eqref{gloabl-solution-middle-3}, this
gives a uniform local oscillation bound for $u_i$. Fix $p_0\in M$ and
set $c_i=u_i(p_0)$,
then $u_i-c_i$ is locally uniformly bounded and equi-Lipschitz.
After passing to a subsequence, $u_i-c_i$ converges locally uniformly
to a function $u$. By the Compactness Theorem~2.1 of
\cite{Huisken2001TheIM}, $u$ is a weak solution of IMCF.

Fix a $r_0<\epsilon_2$,
\begin{equation}
    u_{i}(\delta_i,\vec{\theta})-u_i(r_0,\vec{\theta})\leq u_{+,i}(\delta_i,\vec{\theta})-u_{-,i}(r_0,\vec{\theta})=\frac{4}{3}\log (\frac{\delta_i-C\delta_i^2}{r_0-Cr_0^2})
\end{equation}
which goes to $-\infty$ as $i\rightarrow \infty$. This estimate also holds for $u$, and proves the first part of (\ref{global-existence-ends-estimate}). And the second part is proved similarly by using (\ref{bar-u-infinity}).

\noindent \textit{iii)}

 By (\ref{IMCF-initial-H}), 
\begin{equation}\label{hawking-initial-2}
    m_H(\Sigma_{i,0})=m_H(S(\delta_i))=-\frac{1}{36\pi^{3/2}}|h_1|^{3/2}+O(\delta_i)
\end{equation}

Finally, by \eqref{Geroch-monotonicity}\eqref{hawking-initial-2} and lower semicontinuity of $\int H^2$, in order to prove \eqref{estimate-Hawking}, it suffices to show that
\begin{equation}\label{chi-estimate}
\chi(\Sigma_t)\leq 0.
\end{equation}
Indeed, suppose that \(\Sigma_t\) has a spherical component \(\Gamma\). Since \(M\) is irreducible, \(\Gamma\) bounds an embedded \(3\)-ball \(B\subset M\). By Theorem~\ref{IMCF-existence}(3), \(E_t\) is connected. Moreover, by construction, \(E_t\) contains the inner end of \(M\), which is diffeomorphic to \((0,1)\times T^2\). Since \(B\) is compact, \(E_t\) cannot be contained in \(B\). Hence \(E_t\) lies on the exterior side of \(\Gamma\), so that \(B\) is a compact connected component of \(M\setminus E_t\). This contradicts Theorem~\ref{IMCF-existence}(3).

Therefore, \(\Sigma_t\) has no spherical component. Since each component of \(\Sigma_t\) is a closed orientable surface, it follows that $\chi(\Sigma_t)\leq 0$.

\end{proof}

We are now ready to prove the main theorem, Theorem~\ref{thm-main1}. The proof is divided into three parts. First, we derive an oscillation estimate for $\Sigma_t$ using the gradient estimate together with the sub- and supersolutions. Second, using this oscillation estimate, we obtain an estimate for $m_H(\Sigma_t)$ as $t\to\infty$. Finally, we prove the rigidity statement.

The rigidity argument follows the strategy used in the proof of the rigidity case of the Penrose inequality in \cite{Huisken2001TheIM}. The main difference is that we prove $\Sigma_{t}^{+} = \Sigma_{t}$ by exploiting the topology of $M$, whereas \cite{Huisken2001TheIM} assumes the absence of interior minimal surfaces. We include the details for completeness.

\begin{proof}
\noindent
In the proof, we fix the global weak IMCF given by \textbf{Proposition \ref{global-existence}}.

\noindent\textit{i) Oscillation of $\Sigma_t$}

Let
 \begin{equation}
 A(t)=\sup \{r(x):x\in u^{-1}(t)\}, \quad a(t)=\inf \{r(x):x\in u^{-1}(t)\}
 \end{equation}
denote the outer and inner radii of $\Sigma_t$, respectively. When $r$ is sufficiently large, $\log r$ is a subsolution. Thus, by \textbf{Lemma \ref{IMCF-comparison}}, we have
\begin{equation}\label{IMCF-comparison-1}
A(t+\tau)\leq \exp(\tau)A(t)
\end{equation}
for sufficiently large $t$ and arbitrary $\tau>0$.

Fix a sufficiently large $t_0$ and choose $p\in u^{-1}(t_0)$ such that $r(p)=a(t_0)$. By \eqref{estiamte-sigma}, we may take $s=\frac{1}{1000}r(p)$ in \eqref{IMCF-gradient-estimate}. Then $\Sigma_0\cap B(p,\frac{1}{1000}r(p))=\emptyset$ and $S(r(p))\subset B(p,\frac{1}{1000}r(p))$. Therefore,
\begin{equation}
|\nabla u|(r,\vec{\theta}) \leq \frac{C}{r} \qquad \text{on } S(r(p)).
\end{equation}
Furthermore, since $\operatorname{Diam}S(r(p))<2r(p)^{\frac{2}{3}}\operatorname{Diam}(h)$, it follows that
\begin{equation}\label{IMCF-os-estimate-1}
    \sup \{|u(q)-u(q')|:q,q'\in S(r(p))\}\leq \frac{C}{r(p)^{1/3}}
\end{equation}

Therefore, $u-t_0\geq -Cr(p)^{-1/3}$ on $S(r(p))$, and hence $u^{-1}(t_0-Cr(p)^{-1/3})$ lies inside $S(r(p))$. Combining this with \eqref{IMCF-comparison-1}, we obtain
\begin{equation}\label{IMCF-os-estimate-2}
A(t_0)\leq \exp\left(\frac{C}{r(p)^{1/3}}\right)A\left(t_0-\frac{C}{r(p)^{1/3}}\right) \leq \exp\left(\frac{C}{r(p)^{1/3}}\right)a(t_0).
\end{equation}

\noindent\textit{(ii) Estimeta $\int H^2$}

For a small fixed $\epsilon$, choose $t_1$ sufficiently large so that $u_{-,\epsilon}$ and $u_{+,\epsilon}$ are respectively a subsolution and a supersolution in the region $\{t\geq t_1\}$. Recall that $|\nabla u|=H$ almost everywhere on $\Sigma_t$ for almost every $t$. By the coarea formula,
\begin{equation}
   \begin{aligned}
\int_{t_1}^{t_2}\left(\int_{\Sigma_t} H^2 d A\right) d t&=\int_{u^{-1}\left(t_1, t_2\right)} |\nabla u|^3 d v_g\\
& \geq \int_{A(t_1)\leq r\leq a(t_2)}|\nabla u|^3 d v_g.
   \end{aligned}
\end{equation}
If $|\nabla u|=0$ in some region, corresponding to a jump of the weak IMCF, this region contributes nothing to the right-hand side, so the inequality remains valid.

Next, choose $t_1$ sufficiently large and let
\begin{equation}\label{choice-t2}
   t_2=\inf \{u(p):p\in S(a(t_1)+a(t_1)^{3/4})\}.
\end{equation}
Thus $t_2$ is the first time at which the weak IMCF reaches $S(a(t_1)+a(t_1)^{3/4})$. By increasing $t_1$ if necessary, we may assume that $dv_g\geq (1-\epsilon)a(t_1)^{4/3}drdh_2$. Therefore

\begin{equation}\label{IMCF-estimate-middle}
   \begin{aligned}
\int_{u^{-1}\left(t_1, t_2\right)} |\nabla u|^3 d v_g
& \geq \int_{A(t_1)\leq r\leq a(t_2)}|\nabla u|^3 d v_g\\
&\geq (1-\epsilon)a(t_1)^{4/3}\int_{T^2}\int_{A(t_1)}^{a(t_2)} |\nabla u(r,\vec{\theta})|^3drdh_2\\
&\geq (1-\epsilon)\frac{a(t_1)^{4/3}}{(a(t_2)-A(t_1))^2}\int_{T^2}\bigg( \int_{A(t_1)}^{a(t_2)}|\nabla u(r,\vec{\theta})|dr \bigg)^3dh_2\\
&\geq (1-2\epsilon)\frac{a(t_1)^{4/3}}{(a(t_2)-A(t_1))^2}\int_{T^2}\bigg(u(a(t_2),\vec{\theta})-u(A(t_1),\vec{\theta})\bigg)^3dh_2.
   \end{aligned}
\end{equation}
Here we used H\"older's inequality and $|\partial_r u|\leq (1+o(1))|\nabla u|$. It remains to estimate the integrand on the right-hand side.

By the definition of $a(t_2)$, there exists $p\in S(a(t_2))$ such that $u(p)=t_2$. Applying \eqref{IMCF-os-estimate-1}, we obtain
\begin{equation}
    |u(a(t_2),\vec{\theta})-t_2|\leq \frac{C}{a(t_2)^{1/3}}\leq \frac{C}{a(t_1)^{1/3}}\quad \forall \vec{\theta}\in T^2.\label{estimate-delta-u-1}
\end{equation}
Similarly,
\begin{equation}\label{estimate-delta-u-2}
    |u(A(t_1),\vec{\theta})-t_1|\leq\frac{C}{a(t_1)^{1/3}} \quad \forall \vec{\theta}\in T^2.
\end{equation}

We compare $u_{-,\epsilon}(\cdot)-u_{-,\epsilon}(a(t_1))$ with $u-t_1+Ca(t_1)^{-1/3}$. By \eqref{estimate-delta-u-2}, on $S(a(t_1))$ we have
\begin{equation}
    u_{-,\epsilon}(\cdot)-u_{-,\epsilon}(a(t_1))=0,\quad u-t_1+Ca(t_1)^{-1/3}\geq 0.
\end{equation}
Applying \textbf{Lemma \ref{IMCF-comparison}}, we obtain on $S(a(t_1)+a(t_1)^{3/4})$
\begin{equation}
    \begin{aligned}
    u-t_1+Ca(t_1)^{-1/3}&\geq u_{-,\epsilon}(a(t_1)+a(t_1)^{3/4})-u_{-,\epsilon}(a(t_1))\\
    &=\left(\frac{4}{3}-\epsilon\right)\log \left(1+\frac{1}{a(t_1)^{1/4}}\right).
    \end{aligned}
\end{equation}
Therefore, by the definition of $t_2$ and by taking $t_1$ sufficiently large,
\begin{equation}\label{estimate-delta-u-3}
    t_2-t_1\geq \left(\frac{4}{3}-2\epsilon\right)\frac{1}{a(t_1)^{1/4}}.
\end{equation}

By using $u_{+,2\epsilon}$ instead, we see that
\begin{equation}\label{estimate-delta-u-4}
    t_2-t_1 \leq (\frac{4}{3}+2\epsilon)\frac{1}{a(t_1)^{1/4}}
\end{equation}
From \eqref{estimate-delta-u-1}, \eqref{estimate-delta-u-2}, and \eqref{estimate-delta-u-3}, we obtain, for sufficiently large $t_1$,
\begin{equation}
   \begin{aligned}\label{estimate-delta-u}
u(a(t_2),\vec{\theta})-u(A(t_1),\vec{\theta})&\geq \left(\frac{4}{3}-\epsilon\right)\log \left(1+\frac{1}{a(t_1)^{1/4}}\right)-\frac{C}{a(t_1)^{1/3}}\\
& \geq \left(\frac{4}{3}-2\epsilon\right)\frac{1}{a(t_1)^{1/4}} \quad\forall \vec{\theta}.
   \end{aligned}
\end{equation}

Moreover,
\begin{equation}
\label{estimate-delta-r}
       a(t_2)-A(t_1)\leq \left(a(t_1)+a(t_1)^{3/4}\right)-a(t_1)=a(t_1)^{3/4}.
\end{equation}

Substituting \eqref{estimate-delta-u} and \eqref{estimate-delta-r} into \eqref{IMCF-estimate-middle}, we obtain
\begin{equation}\label{IMCF-estimate-middle-4}
   \int_{t_1}^{t_2}\left(\int_{\Sigma_t} H^2 d A\right) d t\geq (1-100\epsilon)\frac{64}{27}\frac{1}{a(t_1)^{11/12}}|h_2|.
\end{equation}

Using the monotonicity of the Hawking mass, we have
\begin{equation}
\begin{aligned}\label{IMCF-estimate-middle-3}
\int_{t_1}^{t_2} \int_{\Sigma_t} H^2dA_tdt &=\int_{t_1}^{t_2}\frac{1}{|\Sigma_t|^{1/2}}\bigg(|\Sigma_t|^{1/2} \int_{\Sigma_t} H^2dA_t\bigg)dt\\
& \leq -(16\pi)^{3/2}m_H(\Sigma_{t'}) \int_{t_1}^{t_2} \frac{1}{|\Sigma_t|^{1/2}} d t.
\end{aligned}
\end{equation}
where $t'$ is any number less than $t_1$.

For almost every $t\in[t_1,t_2]$, $\Sigma_t$ is a $C^{1,\alpha}$ hypersurface separating the two ends of $M$. Hence it represents the generator of $H_2((C,\infty)\times T^2,\mathbb Z)\cong\mathbb Z$ for some large $C$,
and the projection $\Pi:\Sigma_t\longrightarrow T^2$ has degree one. By Sard's theorem and the area formula, almost every point in $T^2$ is a regular point. Pick a regular point $p$, the determinant is 
\begin{equation}
    dA_{\Sigma_t}\geq (1-o(1))r^{4/3}\sqrt{1+r^{-4/3}|\nabla_{h_2} f|^2}dA_{h_2}
\end{equation}
And therefore 
\begin{equation}
|\Sigma_t|
\geq
(1+o(1))\,a(t_1)^{4/3}|h_2|,
\qquad
t\in[t_1,t_2],
\end{equation}

Put it into (\ref{IMCF-estimate-middle-3}) and use (\ref{estimate-delta-u-4}), we see that
\begin{equation}
    \int_{t_1}^{t_2} \int_{\Sigma t} H^2dA_tdt\leq -\frac{4}{3}(16\pi)^{3/2}(1+10\epsilon+o(1))\frac{1}{|h_2|^{1/2}}\frac{1}{a_{t_1}^{11/12}}m_H(\Sigma_{t'})
\end{equation}
Combine with (\ref{IMCF-estimate-middle-4}) and let $t_1\rightarrow \infty$, and note that $\epsilon$ can be made arbitraraly small, we see that
\begin{equation}\label{proof-main-middle-1}
    \frac{1}{36\pi^{3/2}}|h_2|^{3/2}\leq -m_H(\Sigma_{t'})\leq \frac{1}{36\pi^{3/2}}|h_1|^{3/2} \quad \forall t'
\end{equation}
    and the last inequality follows from (\ref{estimate-Hawking}).

\noindent\textit{(iii) Rigidity}

If $|h_1|=|h_2|$, from (\ref{proof-main-middle-1}), $m_{H}(\Sigma_t)$ is constant for all $t$. By \textit{(1)} of \textbf{Theorem \ref{IMCF-existence}}, $H > 0$ a.e. on $\Sigma_{t}$ for a.e. $t$. Combined with the monotonicity (\ref{Geroch-monotonicity}), we have
\begin{equation}\label{main-rigidity-middle-1}
\int_{\Sigma_{t}}|DH|^{2}=0
\end{equation} 
for a.e. $t$. Therefore by \textit{(6)} of \textbf{Theorem \ref{IMCF-existence}} and lower semicontinuity, (\ref{main-rigidity-middle-1}) holds for all $t$. Therefore $\Sigma_{t}$ has constant mean curvature on each connected component. By (\ref{IMCF-gradient-estimate}) and Regularity Theorem in Section 1 of \cite{Huisken2001TheIM}, $\Sigma_t$ has locally uniform $C^{1}$ estimates, it follows by elliptic theory that $\Sigma_{t}$ is smooth for each $t$, with estimates that are locally uniform for $t$. Similar considerations apply to $\Sigma_{t}^{+}$.

If there is a jump at time $t$, then $\Sigma_{t}^{+} \neq \Sigma_{t}$, and $H=0$ on a portion of $\Sigma_{t}^{+}$, so the connected component containing this portion, denoted by $\Sigma_{t,1}^+$, is a minimal surface. Furthermore, \(E_t^+\) is a strictly minimizing hull by \textit{(6)} of \textbf{Theorem \ref{IMCF-existence}}. Note the argument for (\ref{chi-estimate}) actually proves that each connected component of $\Sigma_s$ has $\chi\leq 0$ for each $s$, again by \textit{(6)} and local $C^{1,\alpha}$ estimate in \cite{Huisken2001TheIM} we see that $\chi(\Sigma_{t,1}^+)\leq 0$. Since
$\Sigma_{t,1}^+$ is outer minimizing, it is stable. The stability inequality, together with $R\geq0$ and $\chi(\Sigma_{t,1}^+)\leq0$, implies $\Sigma_{t,1}^+$ is a flat totally geodesic torus. By applying the arguemtent for the rigidity
of area-minimizing tori in manifolds of nonnegative scalar curvature
\cite{CaiGalloway2000} to the region outside $\Sigma_{t,1}^+$, we have the one side local splitting of the metric as a flat product near
$\Sigma_{t,1}^+$. Consequently, there exist nearby tori on the exterior side of $\Sigma_{t,1}^+$ having the same area as $\Sigma_{t,1}^+$, contradicting the
fact that $E_t^+$ is a strictly minimizing hull. Therefore $\Sigma_{t}^{+} = \Sigma_{t}$ for all $t$. (This is the only place where the rigidity proof differ from that in \cite{Huisken2001TheIM}. There they use there's no minimal surface as an assumption.)

This shows $H>0$ for $t>0$ and a convergence argument shows that $H(t)$ is locally uniformly positive for $t>0$. Since no jump happens, by Lemma 2.4 in \cite{Huisken2001TheIM}, for each $t>0$ there is some maximal $T>t$ such that $(\Sigma_{s})_{t \le s < T}$ is a smooth evolution. By the regularity derived above, $\Sigma_{s}$ has uniform space and time derivatives as $s \nearrow T$, so the evolution can be continued smoothly past $t=T$. This shows $T=\infty$ and the entire flow is smooth.

By (\ref{Geroch-monotonicity}), we have
\begin{equation*}
R=0, \quad II=\frac{H}{2}g
\end{equation*}
on each $\Sigma_{t}$ where $II$ is the second fundamental form. Using Jacobi equation for normal flow $\frac{d}{dr}H=-|II|^2-\operatorname{Ric}(\nu,\nu)$ where $\nu$ is outer normal vector, we see $\operatorname{Ric}(\nu,\nu)$ is constant on $\Sigma_t$. And thus by Gauss equation we see that $R(\Sigma_t)$ is constant for each $t$, thus $\Sigma_t$ is flat torus.

Write the metric in the form
\begin{equation}
    g=dr^2+f(r)^{4/3}h
\end{equation}
where $h$ is a fixed flat metric on $T^2$. The mean curvature and Hawking mass of $S(r)$ are computed by
\begin{equation}
    \begin{aligned}
        H(S(r))&=\frac{4}{3}\frac{\dot{f}}{f}\\
        m_H(S(r))&=-\frac{1}{64\pi^{3/2}}|S(r)|^{1/2}\int_{S(r)}H(S(r))^2=-\frac{1}{36\pi^{3/2}}|h|^{3/2}\dot{f}^2
    \end{aligned}
\end{equation}
$m_H(S(r))$ is constant, therefore $f(r)=kr+b$, then a shifting in $r$ and a scaling in $h$ yields $g$ takes form (\ref{model}).
\end{proof}

\begin{remark}
In \eqref{choice-t2}, we make this particular choice of $t_2$ because we want the weak IMCF to evolve from $S(r_0)$ to $S(r_0+r_0^\alpha)$, where we choose $\alpha=3/4$ so that $2/3<\alpha<1$. By \eqref{IMCF-os-estimate-1}, the oscillation of $u$ is bounded by $Cr^{-1/3}$. On the other hand, the time needed for the flow to cross a radial distance of order $r^\alpha$ is of order $r^{\alpha-1}$. Thus we require $\alpha>2/3$ so that $r^{\alpha-1}$ dominates the error term $r^{-1/3}$, as used in \eqref{estimate-delta-u}. We also require $\alpha<1$ in the estimate \eqref{IMCF-estimate-middle-3}, which ensures that
\begin{equation}
\frac{|S(r_1)|}{|S(r_2)|}=1+o(1)
\end{equation}
as $r_0\rightarrow\infty$ whenever $r_0\leq r_1,r_2\leq r_0+r_0^\alpha$.
\end{remark}

\begin{remark}
In \cite{Huisken2001TheIM}, Section~7 is devoted to the asymptotic behavior of weak IMCF in the asymptotically flat setting. However, the blow-down argument used there does not directly apply to our setting. Consider an asymptotically flat $3$-manifold $(\mathbb{R}^3,g')$ together with a weak IMCF $u'$, and define
\begin{equation}
g'_{\lambda}(x)=\lambda^2 g'\left(\frac{x}{\lambda}\right),\quad
u'_{\lambda}(x)=u'\left(\frac{x}{\lambda}\right).
\end{equation}
Then $u'_{\lambda}$ is a weak IMCF with respect to $g'_{\lambda}$. As $\lambda\rightarrow0$, there exist a sequence $\lambda_i\rightarrow0$ and constants $c_i$ such that
\begin{equation}
u'_{\lambda_i}-c_i\longrightarrow 2\log|x|
\end{equation}
locally uniformly on $\mathbb{R}^3\setminus\{0\}$. This blow-down argument works because Euclidean space is a metric cone and is invariant under scaling up to diffeomorphism. In contrast, our model \eqref{model} is not asymptotically conical. The natural scaling in our setting is 
\begin{equation}
g_\lambda(r,\vec{\theta})
=
\lambda^{4/3}g\left(\frac{r}{\lambda},\vec{\theta}\right)
=
d\left(\lambda^{-1/3}r\right)^2+r^{4/3}h.
\end{equation}
Fix $p=(1,\theta_0)$ as a base point and set $s=\lambda^{-1/3}(r-1)$. Then
\begin{equation}
g_\lambda
=
ds^2+\left(1+\lambda^{1/3}s\right)^{4/3}h.
\end{equation}
Hence, as $\lambda\to0$, the pointed manifolds $(M,g_\lambda,p)$ converge locally to the product $(\mathbb{R}\times T^2,ds^2+h)$. Thus, unlike the asymptotically flat case, the blow-down limit does not retain a nontrivial radial cone structure on which one could expect an analogous nontrivial limiting IMCF.
\end{remark}

\section{A Brown-York Type Inequality}

In this part, we first extends (\ref{thm-main1-ineq}) to manifolds which are Lipschitz across a hypersurface, and then prove the Brown-York type inequality (\ref{thm-ineq-brownyork-1}). Finally, we give some examples showing that why the method is limited to case where $(\Sigma,g|_\Sigma)$ is a flat torus.

\subsection{Extension to Metrics with Corners}

\begin{theorem}\label{thm-main1-Lip}
Let $(M^3,g)$ be a three-dimensional, irreducible Riemannian
manifold satisfying \textbf{Condition C}. Suppose that a
smooth, closed, embedded hypersurface $\Sigma$ in the
interior of $M$ separates $M$ into a bounded region
$\Omega$ and an exterior region $E$.
Assume that $g$ is Lipschitz across $\Sigma$ and that its
restrictions to the two regions are smooth up to $\Sigma$.
Let $H_+$ and $H_-$ denote the mean curvatures of $\Sigma$
with respect to the outward unit normals of $\Omega$
and $E$, respectively.
Suppose that $R_g\geq 0$ on both sides of $\Sigma$ and that
\begin{equation}\label{3-H-lip}
    H_++H_->0 \qquad \text{on }\Sigma.
\end{equation}
Then
\begin{equation}\label{thm-main1-ineq-lip}
    |h_1|\geq |h_2|.
\end{equation}
\end{theorem}

\begin{proof}
Consider a small collar neighborhood $U$
of $\Sigma$ in $\overline{\Omega}$ such that
\begin{equation}\label{3-U}
    U\cong[0,\delta)\times\Sigma,
    \qquad g|_{\Omega}=ds^2+h(s),
\end{equation}
where $h(s)$ is a smooth
family of Riemannian metrics on $\Sigma$. $H(s)$
be the mean curvature of the coordinate hypersurface
$\{s\}\times\Sigma$ with respect to $\partial_s$.
After shrinking $\delta$ if necessary, there exists
a constant $C>0$ such that $|H(s)|\leq C \qquad \text{in }U.$

Let $0<\epsilon<\delta$ be sufficiently small.
We perform the conformal change
$g_1=\phi^4g|_{\Omega}$ with
\begin{equation}
    \phi=
    \begin{cases}
        1-\varphi(\epsilon-s), & 0\leq s<\epsilon,\\
        1, & \text{elsewhere},
    \end{cases}
\end{equation}
where $\varphi(t)=\exp(-1/t)$, for $t>0$. Since $\varphi$ can be smoothly extended on $\mathbb{R}$, the
$\phi_+$ is smooth and positive.

In the region $0\leq s<\epsilon$, the scalar curvature
$R_1$ of $g_1$ satisfies
     \begin{equation}
        \begin{aligned} \label{3-sc-conformal}
R_1 &= \frac{1}{\phi^5} 
\left( R \phi - 8 \Delta_g \phi \right)\\
&=\frac{8}{\phi^5}\bigg(-\partial_{ss}\phi-H(s)\partial_s\phi+\frac{1}{8}R\phi \bigg)\\
&\geq \frac{8}{\phi^5}\big( \varphi''(\epsilon-s)-C\dot{\varphi}(\epsilon-s)  \big)\\
        \end{aligned}
    \end{equation}
Since $\varphi'(t)>0$ for $t>0$ and
\[
    \frac{\varphi''(t)}{\varphi'(t)}
    =\frac{1}{t^2}-\frac{2}{t}
    \longrightarrow+\infty
    \qquad\text{as }t\to0^+,
\]
we may choose $\epsilon$ sufficiently small so that
\begin{equation}\label{3-R1}
    R_1>0 \quad\text{for }0\leq s<\epsilon,
    \qquad
    R_1\geq0 \quad\text{elsewhere}.
\end{equation}
Moreover, the mean curvature $H_1$ of $\Sigma$ with
respect to $g_1$ and the outward unit normal satisfies
  \begin{equation}
           H_1 =\frac{1}{\phi^2}(H_++\frac{4}{\phi}\partial_{\nu}\phi)=\frac{1}{(1-\varphi(\epsilon))^2}\bigg(H_+ -\frac{4}{1-\varphi(\epsilon)}\dot{\varphi}(\epsilon)\bigg)
    \end{equation}
where $\nu=-\partial_s$ is the outward unit normal
with respect to the original metric.

Similarly, we can consider $g_2=\phi^4g|_{M\setminus \Omega}$ and apply the same computation so that
 \begin{equation}
    \begin{aligned}\label{3-R2}
    R_2>0 &\text{ in a small neighborhood of }\Sigma; \quad R_2\geq 0 \text{ elsewhere },\\
     & H_2=\frac{1}{(1-\varphi(\epsilon))^2}\bigg(H_- -\frac{4}{1-\varphi(\epsilon)}\dot{\varphi}(\epsilon)\bigg)
    \end{aligned}
 \end{equation}

From (\ref{3-H-lip}), we might further shrink $\epsilon$ so that
\begin{equation}\label{3-H-lip-2}
    H_1+H_2=\frac{1}{(1-\varphi(\epsilon))^2}\bigg(H_++H_{-}-\frac{8}{1-\varphi(\epsilon)}\dot{\varphi}(\epsilon)\bigg)>0
\end{equation}

From (\ref{3-R1})(\ref{3-R2})(\ref{3-H-lip-2}), we can apply Theorem 5 of \cite{Brendle2010DeformationsOT} to $g_1,g_2$ along $\Sigma$ and approximate to get a smooth metric $\tilde{g}$ so that $R(\tilde{g})\geq 0$ and $\tilde{g}$ agrees with $g_1,g_2$ outside a small neighborhood of $\Sigma$ in $\Omega$ and $M\setminus \Omega$ respectively, and thus agrees with $g|_{\Omega},g|_{M\setminus \Omega}$ near $\{0\}\times T^2$ or $\{\infty\}\times T^2$ respectively. Apply \textbf{Theorem \ref{thm-main1}}, and the theorem follows.

\end{proof}

\subsection{Proof of the Brown–York Type Inequality}\label{subsec-BY}\leavevmode\newline

In this subsection we apply Shi-Tam type argument in \cite{Shi2002PositiveMT} to prove \textbf{Theorem \ref{thm-brownyork}}.

Let $(M^3, \bar{g},V)$ be a static triple. We consider a regular foliation or a smooth one-parameter family of embedded hypersurfaces $\{\Sigma_t\}_{t \in [T_1, T_2)}$ evolving by
\begin{equation}\label{eq:hypersurface-flow}
    \frac{\partial X}{\partial t}(x, t) = f(x, t) \nu(x, t), \quad x \in \Sigma_t.
\end{equation}

For each leaf $\Sigma_t$, we denote by:
\begin{itemize}
    \item $h_t = g|_{\Sigma_t}$ the induced Riemannian metric on $\Sigma_t$;
    \item $\bar{II}$ the second fundamental form of $\Sigma_t \subset (M,g)$ w.r.t. to $f^{-1}\partial_t$;
    \item $\sigma_i$ the $i$-th elementary symmetric function of the principal curvatures $(\kappa_1,\kappa_2)$ of $\Sigma_t$ in $(M,g)$ for $i=1,2$
\end{itemize}

The background metric locally splits into
\begin{equation}\label{metric-ambient}
    \bar{g} = f^2 dt^2 + h_t.
\end{equation}

And we define a new metric
\begin{equation}\label{metric-eta}
    g_u = u^2 dt^2 + h_t,
\end{equation}
where $u$ is a positive function, and require that the scalar curvature is preserved, namely:
\begin{equation}\label{prescribed-scalar}
    R(g_u) = \bar{R}.
\end{equation}
From (4.3) of \cite{Lu2017MinimalHA}, if we further assume $\bar{R}=0$, then (\ref{prescribed-scalar}) transfers equation of $u$ as
\begin{equation}\label{Shi-prescribed-scalar}
    \bar{H}(t)\frac{\partial u}{\partial t}=\frac{u^2}{f}\Delta_t u-\frac{R(\Sigma_t)}{2f}u^3+\frac{u}{f}\big(f^2|\bar{II}|^2+f^2\sigma_2+\frac{\partial}{\partial_t}(f\bar{H})\big)
\end{equation}
where $\bar{H}(t)$ is the mean curvature of $\Sigma_t$ in $\bar{g}$, $\Delta_t$ is the Laplacian of $(\Sigma_t,h_t)$ and $R(\Sigma_t)$ is the scalar curvature of $(\Sigma_t,h_t)$.

In the proof of Brown-York type inequalities, in particular in \cite{Lu2017MinimalHA}, a monotonicity formula for weighted total mean curvature is used. More precisely, Proposition 2.2 in \cite{Lu2017MinimalHA} states that

\begin{lemma}\label{lem-brownyork-monotonicity}
Let $(M,\bar{g},V)$ be a static triple, and let $\{\Sigma_t\}$ be a smooth family of hypersurfaces evolving according to~\eqref{eq:hypersurface-flow}. Let $\bar{g}, g_u$ as above such that $R(g_u) = \bar{R}$. Then the weighted total mean curvature satisfies
\begin{equation}\label{eq-monotonicity-formula}
    \frac{d}{dt} \left( \int_{\Sigma_t} V(\bar{H}(t) - H_u(t)) \, dh_t\right) = - \int_{\Sigma_t} \frac{(u - f)^2}{u} \bar{H}(t) \frac{\partial V}{\partial \nu} \, dh_t - \int_{\Sigma_t} V \sigma_2 \frac{(u - f)^2}{u} \, dh_t,
\end{equation}
where $dh_t$ is the area form on $(\Sigma_t, h_t)$, $\bar{H}(t)$ and $H_u(t) = \frac{f}{u}\bar{H}(t)$ denote the mean curvature of $\Sigma_t$ for $g$ and $g_u$ respectively w.r.t $\partial_t$.
\end{lemma}

Next we prove \textbf{Theorem \ref{thm-brownyork}}
\begin{proof}
    In the region $\{r\geq r_0\}$ in $\bar{M}$, consider 

    \begin{equation}
        \begin{aligned}\label{3-metric-deformed}
        \bar{g}&=dr^2+r^{4/3}\bar{h},\quad \text{for } r\geq r_0\\
        g_u&=u^2dr^2+r^{4/3}\bar{h},\quad \text{for }r\geq r_0
        \end{aligned}
    \end{equation}

     We want to find $u$ so that $R(g_u)=0$ and $H(g_u)=H-\epsilon$ on $\Sigma$ where $\epsilon>0$ is an arbitraraly small. Then (\ref{Shi-prescribed-scalar}) becomes
     \begin{equation}\label{3-brownyork-evolution-equ}
        \begin{cases}
               \frac{4r^{1/3}}{3}\frac{\partial u}{\partial r}&=u^2\Delta_{\bar{h}} u \quad \text{for }r>r_0\\
               u&=\frac{4/(3r_0)}{H-\epsilon}\quad \text{at }r=r_0
        \end{cases}
     \end{equation}
     where $\Delta_{\bar{h}}$ is the Laplacian of $(T^2,\bar{h})$.

Define a new time variable $\tau(r)=\frac{9}{8}r^{2/3}$, then $\frac{\partial u}{\partial r} = \frac{\partial u}{\partial \tau} \frac{d\tau}{dr} =\frac{3}{4}r^{-1/3} u_\tau$, and the first line of (\ref{3-brownyork-evolution-equ}) becomes

\begin{equation}\label{3-brownyork-evolution-equ-1}
   u_\tau = u^2 \Delta_{\bar{h}} u 
\end{equation}

By the parabolic maximum principle, there exists a $C>0$ so that
\begin{equation}\label{3-estimate-L-infinity}
\frac{1}{C} \le u(\tau, \vec{\theta}) \le C
\end{equation}

As a result, the operator $u^2 \Delta$ is uniformly elliptic. Standard quasi-linear parabolic theory (see \cite{ladyzhenskaia1968linear}, Th. IV.10.1, or \cite{friedman1964partial} Th.1, Chap. 4 for example) guarantees that the solution exists globally and becomes smooth for $\tau > \frac{9}{8}r_0^{2/3}$. Let $E(\tau) = \int_{T^2} |\nabla u|^2 \, d\bar{h}$ be the energy functional, and take derivative w.r.t to $\tau$,
\begin{equation}
    \begin{aligned}
    \frac{dE}{d\tau} &= 2 \int_{T^2} \nabla_{\bar{h}} u \cdot \nabla_{\bar{h}} u_\tau \, d\bar{h} \\
    &= -2 \int_{T^2} \Delta_{\bar{h}} u \, (u^2 \Delta_{\bar{h}} u) \, d\bar{h} \\
    &= -2 \int_{T^2} u^2 (\Delta_{\bar{h}} u)^2 \, d\bar{h}\leq \frac{-2}{C^2}\int_{T^2}(\Delta_{\bar{h}} u)^2d\bar{h}
    \end{aligned}
\end{equation}
where in the last inequality we used (\ref{3-estimate-L-infinity}).

Let $\lambda_1 > 0$ be the first non-zero eigenvalue of $-\Delta_{\bar{h}}$ on $T^2$. By the Poincar\'{e} inequality for functions with mean zero, we know $\int (\Delta u)^2 \, d\bar{h} \ge \lambda_1 \int |\nabla u|^2 \, d\bar{h}$. Thus:
\begin{equation}
\frac{dE}{d\tau} \le -\frac{2}{C^2} \lambda_1 E(\tau)
\end{equation}

It follows that $E(\tau)$ decays exponentially.
Uniform parabolic regularity then
yield exponential decay of the spatial derivatives
of $u$ in every fixed $C^k$ norm.
By \eqref{3-brownyork-evolution-equ-1},
$u_\tau$ also decays exponentially, so $u$ converges
uniformly to a limit as $\tau\to\infty$, denoted by $u_\infty$.
Consequently, there exists $\alpha>0$ such that

\begin{equation}
    \begin{aligned}\label{3-u-asmp}
        |u(r, x)-u_\infty|+ |\nabla_{\bar{h}} u(r, x)|+|\nabla_{\bar{h}}^2 u(r, x)| = O\left(\exp\left(-\alpha r^{2/3}\right)\right)\\
              |u_r(r, x)| = O\left(r^{-1/3} \exp\left(-\alpha r^{2/3}\right)\right)
    \end{aligned}
\end{equation}

Moreover
\begin{equation}\label{3-constancy}
    \begin{aligned}
       \frac{d}{d\tau} \int_{T^2} \frac{1}{u(\tau, \vec{\theta})} \, d\bar{h} &=\int_{T^2} \frac{d}{d\tau}\frac{1}{u(\tau, \vec{\theta})} \, d\bar{h}\\
         &=\int_{T^2} -\Delta_{\bar{h}} u \, d\bar{h} =0
    \end{aligned}
\end{equation}
thus $\int_{T^2} \frac{1}{u(\tau, \vec{\theta})} \, d\bar{h}$ is constant. Using the initial condition (\ref{3-brownyork-evolution-equ}), the limit $u_{\infty}$ is

\begin{equation}\label{3-u-infty}
   u_{\infty} =\lim_{\tau\rightarrow \infty} u(\tau,\vec{\theta})=\frac{\frac{4}{3r_0}|\bar{h}|}{\int_\Sigma (H-\epsilon) d\bar{h}}
\end{equation}

Consequently, the metric $g_u$ has asymptotic behavior as
\begin{equation}\label{3-gu-asmp}
    g_u\sim u_{\infty}^2dr^2+r^{4/3}\bar{h}=ds^2+s^{4/3}(\frac{1}{u_{\infty}^{4/3}}\bar{h})
\end{equation}
and it satisfies decaying requirement at infinity in (\ref{g-asym0})(\ref{g-asym}) due to the exponential decay. Besides, on $\Sigma$ we have
\begin{equation}
    -H_u+H=\epsilon>0
\end{equation}
satifying (\ref{3-H-lip}). The minus sign arises because $H_u$ is computed with respect to the unit normal $u^{-1}\partial_r$,
which points into the exterior region. Applying Theorem~\ref{thm-main1-Lip}, we can glue $(\Omega,g)$ and $(\{r\geq r_0\}\subset \bar{M},g_u)$ along the common boundary to get $(M',g')$. Since $M'$ is obtained from $\Omega$ by attaching the product end $\{r\geq r_0\}\subset\bar M$, every embedded sphere in $M'$ can be moved into
the interior of $\Omega$ by an ambient isotopy.
Thus, the irreducibility of $\Omega$ implies that
$M'$ is irreducible. Besides since $g'$ agrees with $g,g_u$ away from $\Sigma$, $g'$ satisfies \textbf{Condition C}. Apply \textbf{Theorem \ref{thm-main1-Lip}}, we arrive at 
\begin{equation}\label{3-mass-ineq}
    \frac{1}{u_{\infty}^{4/3}}|\bar{h}|\leq |h|
\end{equation}

Next apply (\ref{eq-monotonicity-formula}), and in our case it reduces to
\begin{equation}
    \frac{d}{dt} \left( \int_{\Sigma_t} V(\bar{H}(t) - H_u(t)) \, dh_t\right) = 0
\end{equation}
Combined with (\ref{3-gu-asmp}), we see that 
\begin{equation}
    \begin{aligned}
    \int_{\Sigma}V\big(\bar{H}-H_u\big)&= \lim_{t\rightarrow \infty} \int_{\Sigma_t}V\big(\bar{H}(t)-H_u(t)\big)dh_t\\
    &=\frac{4}{3}\big(1-\frac{1}{u_{\infty}}\big)|\bar{h}|
    \end{aligned}
\end{equation}
Using (\ref{3-mass-ineq}) and the boundary condition
$H_u=H-\epsilon$ 
\begin{equation}\label{3-brownyork-general-1}
    \int_{\Sigma}V\bigg(\bar{H}-(H-\epsilon)\bigg)\geq \frac{4}{3}\bigg(|\bar{h}|-|\bar{h}|^{1/4}|h|^{3/4}\bigg)
\end{equation}
Note that $V=\frac{1}{r_0^{1/3}}$ on $\Sigma$ and $\int_{\Sigma_t}V\bar{H}(t)dh_t=\frac{4}{3}|\bar{h}|$ for all $t$, letting $\epsilon\rightarrow 0$ yields
\begin{equation}
        \int_{\Sigma} H\,dA_g
    \leq
    \frac{4}{3}r_0^{1/3}|\bar h|^{1/4}|h|^{3/4}.
\end{equation}
Finally, since $(\Sigma,g|_{\Sigma})$ is isometric to the $\{r=r_0\}$ of $\bar{g}=dr^2+r^{4/3}\bar{h}$, we see that $A((\Sigma,g|_{\Sigma}))=r_0^{4/3}|\bar{h}|$, and (\ref{thm-ineq-brownyork-1}) follows.

\end{proof}

\begin{remark}
The inequality \eqref{thm-ineq-brownyork-1}
can also be proved directly without using
Lemma~\ref{lem-brownyork-monotonicity}.
The flatness of the slices $\Sigma_t$ leads to
\eqref{3-u-infty}; combining this identity with
\eqref{3-mass-ineq} and letting $\epsilon\to0$
yields the desired inequality.
We present the argument using the
monotonicity formula to prepare for
the discussion in the next subsection of why
this method fails for general convex surfaces
$\Sigma$.
\end{remark}

\subsection{Limitations of the Method for General Tori}\label{subsection-Failure-in-General}\leavevmode\newline

Suppose we have a positive mass type theorem on a static background $(M,\bar{g},V)$ for some kind of mass $m$. To derive a Brown-York type theorem from it, one usually needs following

\begin{itemize}
    \item[\hypertarget{3-fail-1}{*}] A monotonicity formula for weighted total mean curvature

    \item[\hypertarget{3-fail-2}{**}] A regular foliation of the background static manifold $\Sigma_t$ and $\bar{g}=f^2dt^2+h_t$ so that
    \begin{equation}
    \lim_{t\rightarrow \infty}\int_{\Sigma_t}V\bar{H}(t)=m(\bar{g})
    \end{equation}
    and each $\Sigma_t$ is convex in certain sense.

    \item[\hypertarget{3-fail-3}{***}] A global solution $u$ to (\ref{Shi-prescribed-scalar}) which converges at $\infty$ and 
    \begin{equation}
    \lim_{t\rightarrow \infty}\int_{\Sigma_t}VH(t)=F(m(g_u))
    \end{equation}
    for some function $F$
\end{itemize}

And notations are as in Theorem~\ref{thm-brownyork}.
If these conditions were satisfied—which, as we will demonstrate, is not the case—letting $\epsilon\rightarrow 0$ in (\ref{3-brownyork-general-1}) yields
    \begin{equation}\label{3-brownyork-general}
    \int_{\Sigma}V\big(\bar{H}(t)-H(t)\big)\geq \frac{4}{3}\bigg(|\bar{h}|-|\bar{h}|^{1/4}|h|^{3/4}\bigg)
\end{equation}
From the previous result, one might hope this can be derived from \textbf{Theorem \ref{thm-main1}}. However, \hyperlink{3-fail-1}{*}, \hyperlink{3-fail-2}{**} and \hyperlink{3-fail-3}{***} all fail in general.

In the following, fix a background metric $(M^3,\bar{g})$ as (\ref{model}), and consider a smoothly embedded torus $\Sigma$. Using the same notations as in the previous Section~\ref{subsec-BY}. Further assume that $\bar{II}(\Sigma)>0$ for $\Sigma$.

\subsubsection{Limitation of Monotonicity}\leavevmode\newline

Let $p$ be a inner most point of $\Sigma$, i.e.
\begin{equation}\label{3-fail-innermostpoint}
    r(p)=\inf \{r(q):q\in \Sigma\}
\end{equation}
By the comparison of mean curvature with the slice $\{r=r(p)\}$, we see $\bar{H}\leq \frac{4}{3r(p)}$. And thus at $p$
\begin{equation}
    \sigma_2(\Sigma)\leq \frac{\bar{H}^2}{4}\leq \frac{\bar{H}}{3r(p)}
\end{equation}
Furthermore, at $p$ we have $\nu=\partial_r$, and thus $\partial_\nu V=-\frac{1}{3r^{4/3}}$. Therefore
\begin{equation}
    -\bar{H}\frac{\partial V}{\partial\nu}-V\sigma_2\geq 0
\end{equation}
Choose $\Sigma$ so that strict inequality holds in a small neighborhood $U$, which can be done for a generic choice of $\Sigma$. We might further choose $H$ so that $\bar{H}-H$ is supported in $U$, and thus $u-f$ is also supported in $U$. Then (\ref{eq-monotonicity-formula}) becomes
\begin{equation}
    \frac{d}{dt} \left( \int_{\Sigma} V(\bar{H}(t) - H_u(t)) \, dh_t\right) =  \int_{U} \frac{(u - f)^2}{u}\bigg( -\bar{H}(t) \frac{\partial V}{\partial \nu}-V\sigma_2 \bigg) > 0
\end{equation}
at $t=0$. So the monotonicity fails.

\subsubsection{Limitation of The Foliation}\leavevmode\newline

In \cite{Shi2002PositiveMT}, the authors use the distance function, \(f=1\) in (\ref{metric-ambient}), to construct a foliation in Euclidean space. On the other hand, in \cite{Lu2017MinimalHA}, an inverse curvature flow is considered in the Schwarzschild manifold by choosing $f=\frac{\sigma_1}{\sigma_2}.$

Both ambient spaces can be viewed as warped product manifolds of the form $ds^2+\phi(s)^2\widetilde{h}$, for a suitable metric \(\widetilde{h}\). In this setting, the Ricci curvature in the radial direction is given by
\begin{equation}\label{eq:radial-ricci}
\operatorname{Ric}(\partial_s,\partial_s)
=-(n-1)\frac{\phi''}{\phi}\leq 0.
\end{equation}

This curvature condition plays an important role in deriving a priori estimates for the support function in \cite{Lu2016InverseCF,Lu2017MinimalHA}, and consequently in establishing the long-time existence of the corresponding inverse curvature flows. More specifically, in the parabolic equation (3.23) of \cite{Lu2017MinimalHA}, the assumption \(\phi''\geq0\) provides the appropriate sign needed to apply the maximum principle. Geometrically, positive curvature tends to focus
flow trajectories and may contribute to singularity
formation. Negative curvature, by contrast, tends
to spread them apart, suggesting that the flow may
remain regular and the evolving hypersurfaces may
gradually approach the coordinate slices.

Furtheremore, we provide with an example that the normal flow (i.e. $f=1$) cannot retain smooth for all time. Again consider a inner most point $p=(r_0,\vec{\theta}_0)$ as (\ref{3-fail-innermostpoint}). Since the normal vector for $\Sigma$ at $p$ is $\partial_r$, the trajectory from $p$ is $\gamma(t)=(r_0+t,\vec{\theta}_0)$. Using the evolution equation for mean curvature and \textbf{Proposition \ref{prop-curvature-computation}}, we see that along $\gamma(t)$
\begin{equation}
    \partial_t H\leq -\frac{H^2}{2}-\frac{4}{9(r_0+t)^2}
\end{equation}
Set $s=r_0+t$, and consider $\partial_s u= -\frac{u^2}{2}-\frac{4}{9s^2}$. Its solutions are
\begin{equation}
    u=\frac{2}{3s}\frac{1+2Cs^{1/3}}{1+Cs^{1/3}}
\end{equation}
and as $C$ goes to $0$ or $\infty$ we have two more solutions $u=\frac{2}{3s}$, $u=\frac{4}{3s}$.
Choose $\Sigma$ so that its principal curvatures
at $p$ are positive and
$0<H(p)<\frac{2}{3r_0}$.
The comparison solution satisfying $u(r_0)=H(p)$
then has
$-1/2r_0^{-1/3}<C<0$.
As $s$ increases from $r_0$, this solution
decreases, crosses zero, and tends to $-\infty$
as $s$ tends to $(-1/C)^3$.

Such initial data can be realized by a smooth
embedded torus with sufficiently small positive
principal curvatures at its innermost point.

Although we have not constructed an example
in which the inverse curvature flow
($f=\sigma_1/\sigma_2$) fails in this model,
the normal-flow example suggests that
\eqref{eq:radial-ricci} is more than a technical assumption, but has geometric significance.

\subsubsection{Limitation of Parabolic System}\leavevmode\newline
Let \(f=1\) in \eqref{Shi-prescribed-scalar}. Using the Gauss equation and the Jacobi equation,
\begin{equation}
\sigma_2
=
\frac{R(\Sigma_t)-\bar{R}}{2}
+\overline{\operatorname{Ric}}(\nu,\nu),
\qquad
\partial_t \bar{H}
=
-|\overline{II}|^2
-\overline{\operatorname{Ric}}(\nu,\nu),
\end{equation}
we arrive at
\begin{equation}
\bar{H}\partial_t u
=
u^2\Delta_{\Sigma_t}u
+
\frac{R(\Sigma_t)}{2}(u-u^3).
\end{equation}

To illustrate a potential obstruction, we consider
a simplified model obtained by freezing the geometric
coefficients in the evolution equation, that \(R(\Sigma_t)\) and the Laplacian \(\Delta_{\Sigma_t}\) are independent of \(t\), and that \(\bar H\) is independent of \(x\), so that \(\bar H=\bar H(t)>0\). We also apply a reparametrization for $t$ to eliminate the coefficient \(\bar H\). We therefore consider the simplified equation
\begin{equation}\label{fail-3-middle}
\partial_t u
=
u^2\Delta_\Sigma u
+
\frac{R(\Sigma)}{2}(u-u^3).
\end{equation}
Although this model need not arise from an actual
geometric flow, it exhibits a finite-time blow-up
mechanism that may obstruct the extension of the
argument to non-flat initial tori.

Let \(\lambda_1\) be the first eigenvalue of the operator $L
=
-\Delta_\Sigma+\frac{R(\Sigma)}{2}$,
and let \(\phi>0\) be a corresponding first eigenfunction:
\begin{equation}
-\Delta_\Sigma\phi
+
\frac{R(\Sigma)}{2}\phi
=
\lambda_1\phi.
\end{equation}

 Using the constant function \(1\) as a test function and the Gauss--Bonnet formula, we see that $\lambda_1\leq 0$ with equality if and only if \(\Sigma\) is flat. Therefore we assume $\lambda_1<0$.
Set
\begin{equation}
\underline u(x,t)
=
A(t)\phi(x).
\end{equation}

Let
\begin{equation}
a
:=
\min_\Sigma\phi
>
0,
\qquad
b
:=
\min_\Sigma\frac{R(\Sigma)}{2}.
\end{equation}
Since \(\lambda_1<0\), we obtain
\begin{equation}
\begin{aligned}
\underline u^2\Delta_\Sigma\underline u
+
\frac{R(\Sigma)}{2}
(\underline u-\underline u^3)&=\frac{R(\Sigma)}{2}A\phi
+
A^3\phi^2
\left(
\Delta_\Sigma\phi
-
\frac{R(\Sigma)}{2}\phi
\right)\\
&= \frac{R(\Sigma)}{2}A\phi
-
\lambda_1A^3\phi^3\\
&\geq
\phi
\left(
bA-\lambda_1a^2A^3
\right).
\end{aligned}
\end{equation}

Therefore, if \(A\) solves
\begin{equation}
A'
=
bA-\lambda_1a^2A^3,
\end{equation}
then \(\underline u=A\phi\) is a subsolution of \eqref{fail-3-middle}. Since $-\lambda_1a^2>0$,
for sufficiently large \(A(0)\), the solution of the ODE blows up in finite time.

Therefore, for any non-flat metric on \(T^2\), there exist positive initial data for which the solution of \eqref{fail-3-middle} blows up in finite time.

\section{Negative Mass and Singularities}\label{mass-negative}
In this section, we restrict our attention to three-dimensional manifolds and discuss two possible phenomena associated with negative mass: the formation of a horn singularity of type $2/3$, and a change in the topology of the interior.

\subsection{Zero Area Singularity and Mass}

Let $(M_m,\bar g_m)$ be the spatial Schwarzschild manifold of mass $m$. When $m>0$, the Schwarzschild metric extends across its minimal hypersurface by reflection, and the resulting complete manifold has two asymptotically flat ends. The situation is rather different when $m<0$. In this case, the Schwarzschild metric develops an
$r^{2/3}$-horn singularity at its inner end (see
\cite{BrayJauregui2013,ShiTam2018,DaiSunWang2025Spin}), in the terminology that a horn of type $b$ is modeled on

\begin{equation}
    dr^2+r^{2b}h.
\end{equation}

Interestingly, negative mass appears to require a sufficiently strong degeneration of the geometry near the singular set. This viewpoint is supported by several extensions of the positive mass theorem to singular spaces. In \cite{DaiSunWang2025,DaiSunWang2025Spin}, it is shown that singularities of $b$-horn type with $b\geq 1$ are still compatible with the positive mass theorem under some mild hypotheses.

A related framework for studying singularities
and their contribution to mass is the theory of
\emph{zero area singularities} (ZAS), initiated
by Bray and developed by Bray--Jauregui
\cite{BrayJauregui2013}.

\begin{definition}\label{def-ZAS} Let
$(M,g)$ be a $3$-manifold whose metric is
smooth in the interior but may degenerate along a compact boundary component
$\Sigma\subset\partial M$. The component $\Sigma$ is called a
\textbf{zero area singularity (ZAS)} if, for every sequence of smooth surfaces
$\Sigma_i$ converging to $\Sigma$ in $C^1$,
\[
    |\Sigma_i|_g\longrightarrow 0.
\]

Furthermore, A ZAS $\Sigma$ is called
regular if, in a neighborhood $U$ of $\Sigma$, one can write
\begin{equation}
    g=\varphi^4\bar g,
\end{equation}
where $\bar g,\varphi$ is smooth up to $\Sigma$ and
\begin{equation}
    \varphi|_\Sigma=0,
    \qquad \varphi>0 \text{ on }U\setminus \Sigma,\qquad 
    \bar\nu(\varphi)>0.
\end{equation}
Here $\bar\nu$ denotes the $\bar g$-unit normal pointing into $M$.
\end{definition}
Thus, although $\Sigma$ is topologically a surface, metrically it collapses
to zero area. The singular boundary of the negative-mass Schwarzschild
manifold is the basic example. Then we can define mass for ZAS.

\begin{definition}

The regular mass for a regular ZAS in $(M,g)$ is defined by
\begin{equation}\label{eq:regular-ZAS-mass}
    m_{\mathrm{reg}}(\Sigma)
    =
    -\frac14
    \left(
        \frac1\pi
        \int_{\Sigma}
        \bar\nu(\varphi)^{4/3}\,dA_{\bar g}
    \right)^{3/2}.
\end{equation}
Suppose $\Sigma$ is ZAS (not necessarily regular ZAS). We can find a sequence $\Sigma_i$ that converges to $\Sigma$ in $C^{1}$, and let $\Omega_i$ be the region enclosed by $\Sigma_i$. If $(M,g)$ is AF, let $\varphi_i$ defined in $M\setminus \Omega_i$ so that
\begin{equation}
    \Delta \varphi_i=0, \qquad \varphi_i|_{\Sigma_i}=0;\qquad \varphi_i=1\text{ at }\infty 
\end{equation}
Then $\Sigma_i$ is a regular ZAS for $(M\setminus \Omega_i,g_i=\varphi_i^4g)$. The ZAS mass of $\Sigma$ is defined by
\begin{equation}
    m_{\mathrm{ZAS}}(\Sigma)
    :=
    \sup_{\{\Sigma_n\}}
    \left(
        \limsup_{n\to\infty}
        m_{\mathrm{reg}}(\Sigma_n)
    \right),
\end{equation}
where the supremum is taken over all sequences $\{\Sigma_n\}$ converging
to $\Sigma$ in $C^1$, and $m_{\mathrm{reg}}(\Sigma_n)$ denotes the regular
mass of $\Sigma_n$.
\end{definition}

The following conjecture is proposed in \cite{BrayJauregui2013}.
\begin{conjecture}\label{conj-ZAS-mass}
    Let $(M^3,g)$ is AF with ZAS $\Sigma$, and $R\geq 0$ on $M\setminus \Sigma$. Then 
    \begin{equation}
        m_{ADM}(g)\geq m_{ZAS}(\Sigma).
    \end{equation}

\end{conjecture}

Building on Bray's earlier work, Bray and Jauregui
proved this inequality assuming an auxiliary
conjecture in conformal geometry.
Under the additional assumption that the ZAS
admits a global harmonic resolution, they also
established that equality holds only for the
negative-mass Schwarzschild metric; see
\cite{BrayJauregui2013}.
When $\Sigma$ is connected, Robbins proved
the mass inequality using weak inverse mean
curvature flow, without assuming the auxiliary
conformal conjecture; see \cite{Robbins2010}.

Consider the flat product $(\mathbb R\times T^2,dt^2+h)$. Zhu proved that every complete metric with nonnegative scalar curvature on $\mathbb R\times T^2$ is isometric to a flat product; see \cite{ZhuRigidity}. Its coordinate tori are totally geodesic and thus have zero mean curvature. In our model, by contrast, the coordinate tori have strict positive mean curvature near infinity, while the scalar curvature remains nonnegative. The $2/3$-horn singularity at $r=0$ can therefore be viewed as the price paid for this positive mean curvature at infinity. More precisely, using the conformal representation \eqref{model2}, the singularity ${r=0}$ is a regular ZAS with
\begin{equation}\label{eq:model-ZAS-mass}
m_{\mathrm{reg}}(\{r=0\})
=
-\frac14
\left(
\frac{1}{\pi}3^{-4/3}|h|
\right)^{3/2}
=
-\frac{1}{36}
\left(\frac{|h|}{\pi}\right)^{3/2}.
\end{equation}
The lower-order error terms in the asymptotic expansion of $g$ in \textbf{Theorem \ref{thm-main1}} do not affect the limiting mass. If the singularity at $\{r=0\}$ is a regular ZAS, then, by Definition~\ref{def-mass-F}, the main theorem can equivalently be written as
\begin{equation}
m_F(g) \geq m_{\mathrm{reg}}(\{r=0\}).
\end{equation} And it's formally analogous to the ZAS mass inequality in \textbf{Conjecture \ref{conj-ZAS-mass}}.

Besides, this phenomenon can also be extended to Brown-York type inequalities. Let $(\Omega^3,g)$ be compact and connected,
with $\partial\Omega=\Sigma_O\sqcup\Sigma_S$, where
$\Sigma_S$ consists of zero area singularities and $g$ is
smooth elsewhere, with $R_g\geq0$.
If $\Sigma_O$ is connected, has positive outward mean
curvature $H$, and admits a strictly convex isometric
embedding into $\mathbb R^3$, by assuming Conjecture~\ref{conj-ZAS-mass}, we have
\begin{equation}\label{ineq-BY-ZAS}
    \frac{1}{8\pi}
    \int_{\Sigma_O}(\bar H-H)\,d\sigma
    \geq m_{\mathrm{ZAS}}(\Sigma_S),
\end{equation}
where $\bar H$ is the outward mean curvature of the
Euclidean embedding. Indeed, the Shi--Tam extension construction
\cite{Shi2002PositiveMT} bounds the boundary integral
from below by the ADM mass of the resulting
asymptotically flat extension. Applying
Conjecture~\ref{conj-ZAS-mass} in place of the positive
mass theorem then gives \eqref{ineq-BY-ZAS}.

\subsection{Asymptotic Locally Hyperbolic Case}\leavevmode\newline

For ALH manifolds with toroidal conformal infinity, the natural zero-mass reference metric is the $m=0$ Kottler metric, namely the hyperbolic cusp. When the mass is negative, the most immediate phenomenon is the formation of a singularity. Compared with the asymptotically flat setting, however, additional phenomena may also occur. We begin by examining the structure of the singularities arising in the negative-mass Kottler family.

Consider $g_m$ as in \eqref{def-Kottler} with $m=-a, a>0$. As $r\to0$, $ g_m\sim \frac{r}{2a}\,dr^2+r^2h$. Let $s$ denote the geodesic distance from $r=0$, i.e. $ds=
(r^2+\frac{2a}{r})^{-1/2}dr\sim\sqrt{\frac{r}{2a}}\,dr$
and hence $s
    \sim
    \frac{2}{3\sqrt{2a}}r^{3/2}$. Consequently,
\begin{equation}\label{AH-singularity-asmp}
g_m
\sim
ds^2+
\left(\frac{3\sqrt{2a}}{2}\right)^{4/3}
s^{4/3}h.
\end{equation}
Thus, within the Kottler family, negative mass produces a horn singularity of type $2/3$. Besides, the positive mass theorem of Galloway and Tsang \cite{GallowayTsang2026} implies that a complete ALH manifold with topology $\mathbb R\times T^2$ and scalar curvature $R_g\geq-6$ has nonnegative mass at its toroidal end. This provides an ALH analogue of the ZAS picture discussed above and motivates Theorem~\ref{thm-ALH-ZAS}. We give a short proof, closely following the argument of Lee and Neves \cite{Lee2013ThePI}.

\begin{proof}

For toroidal surfaces in the negative scalar curvature setting, the Hawking mass is defined by
\begin{equation}\label{eq:hawking-mass-ALH}
m_H(\Sigma)
:=
\sqrt{\frac{|\Sigma|}{16\pi}}
\left(
-\frac{1}{16\pi}
\int_\Sigma (H^2-4)\,dA
\right).
\end{equation}

Let $\Sigma_\epsilon=\{r=\epsilon\}\times T^2$ for $\epsilon>0$ sufficiently small. By \eqref{thm-ALH-g-asmp}, direct computation shows
\begin{equation}\label{eq:inner-Hawking}
m_H(\Sigma_\epsilon)
=
\frac{m_1}{(4\pi)^{3/2}}|h_1|^{3/2}.
\end{equation}
for small $\epsilon$. On the other hand, after a change of variable, $\{r=0\}$ is a regular ZAS, and
\begin{equation}\label{eq:ALH-regular-ZAS}
m_{\mathrm{reg}}(\{r=0\})
=
\frac{m_1}{(4\pi)^{3/2}}|h_1|^{3/2},
\end{equation}
by Proposition~13 of \cite{BrayJauregui2013}.

By the same argument as in part \textit{(i)} of the proof of Proposition~\ref{global-existence}, $\Sigma_\epsilon$ is outer-minimizing for sufficiently small $\epsilon$. Hence the weak IMCF starting from $\Sigma_\epsilon$ exists globally by \cite[Lemma~3.1]{Lee2013ThePI}.
 Besides, the Hawking mass is nondecreasing along the flow by Geroch monotonicity; see \cite[Corollary~3.4]{Lee2013ThePI}.

Finally, since $\mu\leq0$, Lemma~3.13 of \cite{Lee2013ThePI} gives
\begin{equation}
\lim_{t\to\infty}m_H(\Sigma_t)\leq
-
\left(
\frac{1}{4\pi}
\int_{T^2}|\frac{\mu}{4}|^{2/3}\,dA_{h_2}
\right)^{3/2} \leq
\frac{1}{4(4\pi)^{3/2}}
\sup_{T^2}\mu\,|h_2|^{3/2}.
\end{equation}
(Note that the mass aspect function in \cite{Lee2013ThePI} equals $\mu/4$ under our convention, which accounts for the additional factor $1/4$ in the following estimate.) Combining this with \eqref{eq:inner-Hawking} and \eqref{eq:ALH-regular-ZAS} proves the result.

\end{proof}

Negative mass need not force a singularity when the topology
of the interior is allowed to change. The Horowitz--Myers
soliton in Example~\ref{example-HM} has negative mass and
extends smoothly over $B^2\times S^1$: the $\xi$-circle
collapses, while the $\theta$-circle remains nontrivial.
This collapse provides a mechanism for negative mass
without singularities. Indeed, the inward principal
curvatures near conformal infinity satisfy
\[
-\kappa_\xi=-1-x^3+O(x^6),
\qquad
-\kappa_\theta=-1+\frac12x^3+O(x^6).
\]
Since
\begin{equation}\label{eq:H-mass-aspect}
-H=-2+\frac12\mu x^3+o(x^3),
\end{equation}
the negativity of $\mu$ comes from the coefficients of $x^3$ in $-\kappa_{\xi}$, which is consistent with
the collapsing of $\xi$-direction.

As shown in \cite{BrendleHungHM,BrendleHungRigidity}, the Horowitz--Myers soliton is the ground state of mass for ALH manifolds on $B^2\times S^1$. One might ask what might happen if the mass breaks this bound. A possible singular degeneration can be
seen explicitly in a doubly warped model. Write
\[
g_{\mathrm{HM}}
=dt^2+u_0(t)^2d\xi^2+v_0(t)^2d\theta^2,
\]
where $u_0(t)\sim 3t/2$ and $v_0(t)\to v_0(0)>0$
as $t\to0$. Keeping $u_0$ fixed, consider
\begin{equation}\label{g-double-warped2}
g=dt^2+u_0(t)^2d\xi^2+v(t)^2d\theta^2
\end{equation}
with the same conformal boundary metric and radial
normalization. Its scalar curvature and mean curvature of the $\{t=const\}$ slice are given by
\begin{equation}\label{R-double-warped}
R_g=-2\left(
\frac{u_0''}{u_0}+\frac{v''}{v}
+\frac{u_0'v'}{u_0v}
\right);\quad H=\frac{u_0'}{u_0}+\frac{v'}{v}
\end{equation}
Setting $w=v/v_0$ and using $R(g_{\mathrm{HM}})=-6$, then $R\geq -6$ becomes
\begin{equation}\label{R-double-warped-solution}
\begin{aligned}
    0
    &\geq  (wv_0)''+\frac{u_0'}{u_0}(wv_0)'
       +\left(\frac{u_0''}{u_0}-3\right)wv_0\\
    &= v_0w''+
       \left(2v_0'+\frac{u_0'}{u_0}v_0\right)w'+\big(v_0''+\frac{u_0'}{u_0}v_0'+(\frac{u_0''}{u_0}-3)v_0\big)w\\
       &=v_0w''+
       \left(2v_0'+\frac{u_0'}{u_0}v_0\right)w'=\frac{(u_0v_0^2w')'}{u_0v_0}
\end{aligned}
\end{equation}
wherever $v>0$.

If $m_{\mathrm{ALH}}(g)<m_{\mathrm{ALH}}(g_{\mathrm{HM}})$,
then \eqref{eq:H-mass-aspect} and
\eqref{R-double-warped} imply $w'(T)>0$
for sufficiently large $T$. Hence, with
$C=u_0(T)v_0(T)^2w'(T)>0$,
\begin{equation}
w(t)\leq w(T)-C\int_t^T
\frac{d\tau}{u_0(\tau)v_0(\tau)^2},
\qquad 0<t<T,
\end{equation}
As $t \searrow 0$, we have
$u_0(t)\sim 3/2t$
and $v_0(t)\to v_0(0)>0$, so the integral diverges to $+\infty$.
Consequently, there is a $T_1\in(0,T)$
such that $w(T_1)=0$. Thus $v$ remains positive
for $T_1< t<T$ and vanishes at $t=T_1$. Here the $\theta$-circle collapses, whereas the $\xi$-circle collapses in the Horowitz--Myers soliton. And since $u_0'(t)>0$ for all $t$, the smoothness condition in Page 13 of \cite{Petersen2006} fails.

To find a model for the singularity, assume now that $R=-6$. Equation~\eqref{R-double-warped-solution} then holds with equality, so $u_0v_0^2w'$ is constant. Since $v(T_1)=0$, we obtain

\begin{equation}
    v'(T_1)
    =\frac{C}{u_0(T_1)v_0(T_1)}>0.
\end{equation}
Therefore $v(t)\sim k(t-T_1)$ for some $k$. Besides, $u_0'>0$ for all $t$, thus a simplified local
model for this type of collapse is
\begin{equation}
    g_S=d\rho^2+\rho^2d\theta^2
        +(a+b\rho)^2d\xi^2,\quad a,b>0
\end{equation}
The positivity of $b$ contributes to the singularity. The areas of the coordinate tori tend to zero,and this degeneration gives a zero area singularity.

Besides, similarly to the Horowitz--Myers soliton, there exist smooth NNSC metrics on $\mathbb R^2\times S^1$ which agree with \eqref{model} near infinity. In this case, one of the circle factors collapses smoothly in the interior, the same mechanism for the negativity of mass for Horowitz-Myers soliton that prevents the development of singularities.

These examples suggest a general principle: when a geometric inequality fails, singularities may arise in the interior, and their strength may be related to how much the inequality is violated. Bray's notion of ZAS mass may also help us understand this phenomenon beyond the asymptotically flat setting. It would be interesting to explore how this approach extends to other settings.

\bibliographystyle{plainnat}  
\bibliography{ref}

@article{Gerhardt1990FlowON,
  title={Flow of nonconvex hypersurfaces into spheres},
  author={Claus Gerhardt},
  journal={Journal of Differential Geometry},
  year={1990},
  volume={32},
  pages={299-314},
}

@article{Zhou2016InverseMC,
  title={Inverse Mean Curvature Flows in Warped Product Manifolds},
  author={Hengyu Zhou},
  journal={The Journal of Geometric Analysis},
  year={2016},
  volume={28},
  pages={1749-1772},
}

@article{Lu2016InverseCF,
  title={Inverse curvature flow in anti-de Sitter-Schwarzschild manifold},
  author={Siyuan Lu},
  journal={Communications in Analysis and Geometry},
  year={2019},
}

@article{Urbas1990OnTE,
  title={On the expansion of starshaped hypersurfaces by symmetric functions of their principal curvatures},
  author={John I. E. Urbas},
  journal={Mathematische Zeitschrift},
  year={1990},
  volume={205},
  pages={355-372},
}

@article{Huisken2001TheIM,
  title={The inverse mean curvature flow and the Riemannian Penrose Inequality},
  author={Gerhard Huisken and Tom Ilmanen},
  journal={Journal of Differential Geometry},
  year={2001},
  volume={59},
  pages={353-437}, 
}

@misc{Topping2006LecturesOT,
  title  = {Lectures on the Ricci Flow},
  author = {Peter M. Topping},
  year   = {2006},
  note   = {Lecture notes},
}

@article{Greene1979CinftyA,
  title={{$C^\infty $} approximations of convex, subharmonic, and plurisubharmonic functions},
  author={Robert E. Greene and H. Wu},
  journal={Annales Scientifiques De L Ecole Normale Superieure},
  year={1979},
  volume={12},
  pages={47-84},
}

@article{Brendle2010DeformationsOT,
  title={Deformations of the hemisphere that increase scalar curvature},
  author={Simon Brendle and Fernando C. Marques and Andre' Neves},
  journal={Inventiones mathematicae},
  year={2010},
  volume={185},
  pages={175-197},
}

@article{Shi2002PositiveMT,
  title={Positive Mass Theorem and the Boundary Behaviors of Compact Manifolds with Nonnegative Scalar Curvature},
  author={Yuguang Shi and Luen-Fai Tam},
  journal={Journal of Differential Geometry},
  year={2002},
  volume={62},
  pages={79-125},
}

@article{Lu2017MinimalHA,
  title={Minimal hypersurfaces and boundary behavior of compact manifolds with nonnegative scalar curvature},
  author={Siyuan Lu and Pengzi Miao},
  journal={Journal of Differential Geometry},
  year={2017},
}

@book{ladyzhenskaia1968linear,
  title={Linear and quasi-linear equations of parabolic type},
  author={Ladyzhenskaia, O. A. and Solonnikov, V. A. and Uralceva, N. N.},
  volume={23},
  series={Translations of Mathematical Monographs},
  year={1968},
  publisher={American Mathematical Society},
  address={Providence, R.I.}
}

@book{friedman1964partial,
  title={Partial differential equations of parabolic type},
  author={Friedman, Avner},
  year={1964},
  publisher={Prentice Hall},
  address={Englewood Cliffs, N.J.}
}

@article{ChruscielHerzlich2003,
author = {Chru{'s}ciel, Piotr T. and Herzlich, Marc},
title = {The mass of asymptotically hyperbolic {R}iemannian manifolds},
journal = {Pacific Journal of Mathematics},
volume = {212},
number = {2},
pages = {231--264},
year = {2003},
doi = {10.2140/pjm.2003.212.231}
}

@article{Wang2001AHMass,
author = {Wang, Xiaodong},
title = {The mass of asymptotically hyperbolic manifolds},
journal = {Journal of Differential Geometry},
volume = {57},
number = {2},
pages = {273--299},
year = {2001},
doi = {10.4310/jdg/1090348112}
}

@article{Bartnik1986,
  author  = {Bartnik, Robert},
  title   = {The Mass of an Asymptotically Flat Manifold},
  journal = {Communications on Pure and Applied Mathematics},
  volume  = {39},
  pages   = {661--693},
  year    = {1986}
}

@article{GallowayWoolgar2015,
  author  = {Galloway, Gregory J. and Woolgar, Eric},
  title   = {On Static Poincar{\'e}--Einstein Metrics},
  journal = {Journal of High Energy Physics},
  volume  = {2015},
  number  = {6},
  pages   = {051},
  year    = {2015},
  doi     = {10.1007/JHEP06(2015)051}
}

@article{HorowitzMyers1998,
  author  = {Horowitz, Gary T. and Myers, Robert C.},
  title   = {The {AdS/CFT} Correspondence and a New Positive Energy
             Conjecture for General Relativity},
  journal = {Physical Review D},
  volume  = {59},
  pages   = {026005},
  year    = {1998},
  doi     = {10.1103/PhysRevD.59.026005}
}

@article{Birmingham1999,
  author  = {Birmingham, Danny},
  title   = {Topological Black Holes in Anti-de Sitter Space},
  journal = {Classical and Quantum Gravity},
  volume  = {16},
  number  = {4},
  pages   = {1197--1205},
  year    = {1999},
  doi     = {10.1088/0264-9381/16/4/009},
}

@article{GallowaySuryaWoolgar2003,
  author  = {Galloway, Gregory J. and Surya, Sumati and Woolgar, Eric},
  title   = {On the Geometry and Mass of Static, Asymptotically
             {AdS} Spacetimes, and the Uniqueness of the {AdS} Soliton},
  journal = {Communications in Mathematical Physics},
  volume  = {241},
  pages   = {1--25},
  year    = {2003},
  doi     = {10.1007/s00220-003-0912-7},
}

@article{DaiSunWang2025Spin,
  author  = {Dai, Xianzhe and Sun, Yukai and Wang, Changliang},
  title   = {Positive mass theorem for asymptotically flat spin manifolds
             with isolated conical singularities},
  journal = {Transactions of the American Mathematical Society},
  volume  = {378},
  number  = {4},
  pages   = {2617--2642},
  year    = {2025},
  doi     = {10.1090/tran/9331}
}

@article{DaiSunWang2025,
  author  = {Dai, Xianzhe and Sun, Yukai and Wang, Changliang},
  title   = {The positive mass theorem for asymptotically flat manifolds
             with isolated conical singularities},
  journal = {Science China Mathematics},
  volume  = {68},
  number  = {7},
  pages   = {1671--1686},
  year    = {2025},
  doi     = {10.1007/s11425-024-2325-6}
}

@article{BrayJauregui2013,
  author  = {Bray, Hubert L. and Jauregui, Jeffrey L.},
  title   = {A geometric theory of zero area singularities in general relativity},
  journal = {Asian Journal of Mathematics},
  volume  = {17},
  number  = {3},
  pages   = {525--560},
  year    = {2013},
  doi     = {10.4310/AJM.2013.v17.n3.a6}
}

@article{Robbins2010,
  author  = {Robbins, Nicholas P.},
  title   = {Zero area singularities in general relativity and inverse mean curvature flow},
  journal = {Classical and Quantum Gravity},
  volume  = {27},
  number  = {2},
  pages   = {025011},
  year    = {2010},
  doi     = {10.1088/0264-9381/27/2/025011}
}

@article{ShiTam2018,
  author  = {Shi, Yuguang and Tam, Luen-Fai},
  title   = {Scalar curvature and singular metrics},
  journal = {Pacific Journal of Mathematics},
  volume  = {293},
  number  = {2},
  pages   = {427--470},
  year    = {2018},
  doi     = {10.2140/pjm.2018.293.427}
}

@article{GallowayTsang2026,
  author        = {Galloway, Gregory J. and Tsang, Tin-Yau},
  title         = {Positive mass theorems for manifolds with {ALH} toroidal ends},
  journal       = {arXiv preprint},
  year          = {2026},
  eprint        = {2602.08789},
  archivePrefix = {arXiv},
  primaryClass  = {math.DG}
}

@article{BrendleHungHM,
  author        = {Brendle, Simon and Hung, Pei-Ken},
  title         = {Systolic inequalities and the Horowitz--Myers conjecture},
  journal       = {Acta Mathematica},
  note          = {to appear},
  year          = {2026},
  eprint        = {2406.04283},
  archivePrefix = {arXiv},
  primaryClass  = {math.DG}
}

@article{BrendleHungRigidity,
  author        = {Brendle, Simon and Hung, Pei-Ken},
  title         = {The rigidity statement in the Horowitz--Myers conjecture},
  year          = {2025},
  eprint        = {2504.16812},
  archivePrefix = {arXiv},
  primaryClass  = {math.DG}
}

@article{ZhuRigidity,
  author  = {Zhu, Jintian},
  title   = {Rigidity results for complete manifolds with nonnegative scalar curvature},
  journal = {Journal of Differential Geometry},
  volume  = {125},
  number  = {3},
  pages   = {623--644},
  year    = {2023},
  doi     = {10.4310/JDG/1701804153}
}

@article{HuangJang2022,
  author  = {Huang, Lan-Hsuan and Jang, Hyun Chul},
  title   = {Scalar Curvature Deformation and Mass Rigidity for ALH Manifolds with Boundary},
  journal = {Transactions of the American Mathematical Society},
  volume  = {375},
  number  = {11},
  pages   = {8151--8191},
  year    = {2022},
  doi     = {10.1090/tran/8755}
}

@article{SchoenYau1979,
author = {Schoen, Richard and Yau, Shing-Tung},
title = {On the Proof of the Positive Mass Conjecture in General Relativity},
journal = {Communications in Mathematical Physics},
volume = {65},
number = {1},
pages = {45--76},
year = {1979},
doi = {10.1007/BF01940959}
}

@article{SchoenYau1981,
author = {Schoen, Richard and Yau, Shing-Tung},
title = {Proof of the Positive Mass Theorem. II},
journal = {Communications in Mathematical Physics},
volume = {79},
number = {2},
pages = {231--260},
year = {1981},
doi = {10.1007/BF01942062}
}

@article{Witten1981,
author = {Witten, Edward},
title = {A New Proof of the Positive Energy Theorem},
journal = {Communications in Mathematical Physics},
volume = {80},
number = {3},
pages = {381--402},
year = {1981},
doi = {10.1007/BF01208277}
}

@article{ADM1962,
  title={Republication of: The dynamics of general relativity},
  author={Richard L. Arnowitt and Stanley Deser and Charles W. Misner},
  journal={General Relativity and Gravitation},
  year={2004},
  volume={40},
  pages={1997-2027},
  url={https://api.semanticscholar.org/CorpusID:14054267}
}

@article{ChruscielSimon2001,
  author  = {Chru{\'s}ciel, Piotr T. and Simon, Walter},
  title   = {Towards the classification of static vacuum spacetimes
             with negative cosmological constant},
  journal = {Journal of Mathematical Physics},
  volume  = {42},
  number  = {4},
  pages   = {1779--1817},
  year    = {2001},
  doi     = {10.1063/1.1340869}
}

@article{AlaeeHungKhuri2025,
  author  = {Alaee, Aghil and Hung, Pei-Ken and Khuri, Marcus},
  title   = {Boundary Behavior of Compact Manifolds with Scalar Curvature
             Lower Bounds and Static Quasi-Local Mass of Tori},
  journal = {Proceedings of the American Mathematical Society},
  volume  = {153},
  number  = {5},
  pages   = {2153--2167},
  year    = {2025},
  doi     = {10.1090/proc/17223}
}

@article{ChruscielGallowayNguyenPaetz2018,
author = {Chru{'s}ciel, Piotr T. and Galloway, Gregory J.
and Nguyen, Luc and Paetz, Tim-Torben},
title = {On the Mass Aspect Function and Positive Energy Theorems
for Asymptotically Hyperbolic Manifolds},
journal = {Classical and Quantum Gravity},
volume = {35},
number = {11},
pages = {115015},
year = {2018},
doi = {10.1088/1361-6382/aabed1}
}

@misc{BiHaoHeShiZhu2026,
  author        = {Bi, Yuchen and Hao, Tianze and He, Shihang and Shi, Yuguang and Zhu, Jintian},
  title         = {A Proof for the Riemannian Positive Mass Theorem up to Dimension 19},
  year          = {2026},
  eprint        = {2603.02769},
  archivePrefix = {arXiv},
  primaryClass  = {math.DG},
  doi           = {10.48550/arXiv.2603.02769}
}

@misc{BrendleWang2026,
  author        = {Brendle, Simon and Wang, Yipeng},
  title         = {A Dimension Descent Scheme for the Positive Mass Theorem in Arbitrary Dimension},
  year          = {2026},
  eprint        = {2604.08473},
  archivePrefix = {arXiv},
  primaryClass  = {math.DG},
  doi           = {10.48550/arXiv.2604.08473}
}

@article{Neves2010IMCF,
  author  = {Neves, Andr{\'e}},
  title   = {Insufficient convergence of inverse mean curvature flow on asymptotically hyperbolic manifolds},
  journal = {Journal of Differential Geometry},
  volume  = {84},
  number  = {1},
  pages   = {191--229},
  year    = {2010},
  doi     = {10.4310/jdg/1271271798}
}

@article{CaiGalloway2000,
  author  = {Cai, Mingliang and Galloway, Gregory J.},
  title   = {Rigidity of area minimizing tori in 3-manifolds of nonnegative scalar curvature},
  journal = {Communications in Analysis and Geometry},
  volume  = {8},
  number  = {3},
  pages   = {565--573},
  year    = {2000},
  doi     = {10.4310/CAG.2000.v8.n3.a6}
}

@misc{BiZhu2026,
  author        = {Bi, Yuchen and Zhu, Jintian},
  title         = {Riemannian Penrose Inequality in All Dimensions},
  year          = {2026},
  eprint        = {2605.00680},
  archivePrefix = {arXiv},
  primaryClass  = {math.DG}
}

@misc{Wang2025Neumann,
  author        = {Wang, Zhixin},
  title         = {Neumann Data and Second Variation Formula of Renormalized Area
                   for Conformally Compact Static Spaces},
  year          = {2025},
  eprint        = {2504.15048},
  archivePrefix = {arXiv},
  primaryClass  = {math.DG}
}

@article{Lee2013ThePI,
  title={The Penrose Inequality for Asymptotically Locally Hyperbolic Spaces with Nonpositive Mass},
  author={Dan A. Lee and Andr'e Neves},
  journal={Communications in Mathematical Physics},
  year={2013},
  volume={339},
  pages={327-352},
  url={https://api.semanticscholar.org/CorpusID:118287029}
}

@book{Petersen2006,
  author    = {Petersen, Peter},
  title     = {Riemannian Geometry},
  edition   = {2},
  series    = {Graduate Texts in Mathematics},
  volume    = {171},
  publisher = {Springer},
  address   = {New York},
  year      = {2006},
  doi       = {10.1007/978-0-387-29403-2},
  isbn      = {978-0-387-29246-5}
}

\end{document}